\documentclass[12pt, reqno]{amsart}

\usepackage{verbatim}

\usepackage{amssymb}
\usepackage{eucal}
\usepackage{amsmath}
\usepackage{amscd}
\usepackage[dvips]{color}
\usepackage{multicol}
\usepackage[all]{xy}           
\usepackage{graphicx}
\usepackage{bbding}
\usepackage{txfonts}
\usepackage{color}
\usepackage{colordvi}
\usepackage{xspace}
\usepackage{tikz}
\usepackage{setspace}
\usepackage{ifpdf}
\ifpdf
\usepackage[colorlinks,final,backref=page,hyperindex]{hyperref}
\else
\usepackage[colorlinks,final,backref=page,hyperindex,hypertex]{hyperref}
\fi

\usepackage{epstopdf}
\usepackage{epsfig}
\usepackage{pdfsync}    
\usepackage[utf8]{inputenc}

\usetikzlibrary{arrows}

\newcommand{\nc}{\newcommand}
\newcommand{\delete}[1]{}

\newtheorem{theorem}{Theorem}[section]
\newtheorem{prop}[theorem]{Proposition}
\newtheorem{lemma}[theorem]{Lemma}
\newtheorem{coro}[theorem]{Corollary}
\newtheorem{prop-def}[theorem]{Proposition-Definition}

\theoremstyle{definition}
\newtheorem{defn}[theorem]{Definition}
\newtheorem{remark}[theorem]{Remark}

\newtheorem{exam}[theorem]{Example}

\nc{\supcybe}{super CYBE\xspace}    
\nc{\parityrev}{parity reverse\xspace}
\nc{\parityrevs}{parity reverses\xspace}
\nc{\Parityrev}{Parity reverse\xspace}
\nc{\rs}{$r$-matrices\xspace}
\nc{\sr}{super $r$-matrix\xspace}
\nc{\srs}{super $r$-matrices\xspace}

\nc{\pansym}{pan-supersymmetric\xspace}

\nc{\mT}{\mathbf{t}}
\nc{\bfr}{\mathbf{r}}
\nc{\bfT}{\mathbf{T}}
\nc{\Ext}[1]{\widehat{#1}}

\newcommand{\F}{\mathbb{F}}
\newcommand{\op}{$\mathcal {O}$-operator\xspace}
\newcommand{\A}{\mathcal{A}}

\newcommand{\G}{\frak g}
\newcommand{\Z}{\mathbb{Z}}
\newcommand{\ad}{\mathrm{ad}}
\newcommand{\Gl}{\mathfrak{gl}}

\newcommand{\id}{\mathrm{id}}

\newcommand{\ra}{\longrightarrow}
\newcommand{\OO}{\mathcal{O}}

\newcommand{\Sol}{\mathcal{S}}

\newcommand{\I}{\mathbf{I}}
\newcommand{\Hom}{\mathrm{Hom}}

\nc{\tred}[1]{\textcolor{red}{#1}}
\nc{\tblue}[1]{\textcolor{blue}{#1}}
\nc{\tgreen}[1]{\textcolor{green}{#1}}
\nc{\tpurple}[1]{\textcolor{purple}{#1}}
\nc{\btred}[1]{\textcolor{red}{\bf #1}}
\nc{\btblue}[1]{\textcolor{blue}{\bf #1}}
\nc{\btgreen}[1]{\textcolor{green}{\bf #1}}
\nc{\btpurple}[1]{\textcolor{purple}{\bf #1}}

\nc{\rx}[1]{\textcolor{blue}{RX:#1}}
\nc{\cm}[1]{\textcolor{red}{CM:#1}}
\nc{\li}[1]{\textcolor{purple}{#1}}
\nc{\lir}[1]{\textcolor{purple}{Li:#1}}
\nc{\rxd}[1]{{}}
\nc{\cmd}[1]{{}}
\nc{\lid}[1]{{}}
\nc{\lird}[1]{{}}

\nc{\rev}[1]{\textcolor{blue}{#1}}

\begin{document}

\title[Parity duality of super $r$-matrices and $\mathcal {O}$-operators]{Parity duality of super $r$-matrices via $\mathcal O$-operators and pre-Lie superalgebras}

\author{Chengming Bai}

\address{Chern Institute of Mathematics, Nankai University,
Tianjin 300071, P.R. China}

\email{baicm@nankai.edu.cn}

\author{Li Guo}

\address{Department of Mathematics and Computer Science, Rutgers University, Newark, NJ 07102}

\email{liguo@rutgers.edu}

\author{Runxuan Zhang}

\address{School of Mathematics and Statistics, Northeast Normal University, Changchun, Jilin 130024, P.R. China}

\email{zhangrx728@nenu.edu.cn}

\date{\today}

\begin{abstract}
This paper studies super $r$-matrices and operator forms of the super classical Yang-Baxter equation. First by a unified treatment, the classical correspondence between $r$-matrices and $\mathcal{O}$-operators is generalized to a correspondence between homogeneous super $r$-matrices and homogeneous $\mathcal{O}$-operators. Next, by a parity reverse of Lie superalgebra representations, a duality is established between the even and the odd $\mathcal{O}$-operators, giving rise to a parity duality among the induced super $r$-matrices. Thus any homogeneous $\OO$-operator or any homogeneous super $r$-matrix with certain supersymmetry produces a parity pair of super $r$-matrices, and generates an infinite tree hierarchy of homogeneous super $r$-matrices. Finally, a pre-Lie superalgebra naturally defines a parity pair of $\mathcal{O}$-operators, and thus a parity pair of super $r$-matrices.
\end{abstract}

\subjclass[2010]{
17B10, 
17B38, 
16T25,   
17B81, 
81R05 
17B80, 
17B60,
}

\keywords{Lie superalgebra; super classical Yang-Baxter equation; super $r$-matrix; $\mathcal {O}$-operator; relative Rota-Baxter operator; pre-Lie superalgebra}

\maketitle

\tableofcontents

\allowdisplaybreaks

\section{Introduction}
This paper studies super $r$-matrices, that is,  solutions of the super classical Yang-Baxter equation in Lie
superalgebras. A parity duality for its operator forms leads to a duality between even and odd super $r$-matrices. Pre-Lie superalgebras and $\OO$-operators are used to provide both even and odd super $r$-matrices.

\vspace{-.3cm}

\subsection{The CYBE  and its operator approach}
The classical Yang-Baxter equation (CYBE) is the semi-classical
limit of the quantum Yang-Baxter equation. Its independent
importance lies in its crucial role in broad areas such as
symplectic geometry, integrable systems, quantum groups and
quantum field theory (see~\cite{BGN,CP, Dr,Ji,  KS} and the references therein). Its solution has inspired the remarkable works of Belavin, Drinfeld
and others, and the important notions of Lie bialgebras and Manin
triples~\cite{BD}.

Let $\G$ be a Lie algebra and
$r=\sum_{i}x_{i}\otimes y_{i}$ for $x_i, y_i\in \G$.
The tensor form   of the CYBE  is
\begin{equation}\label{eq:1}
    [[r,r]]:=[r_{12}, r_{13}]+[r_{12}, r_{23}]+[r_{13}, r_{23}]=0,
\end{equation}
where
\begin{equation}
    r_{12}:=\sum_{i}x_{i}\otimes y_{i}\otimes 1,
    r_{13}:=\sum_{i}x_{i}\otimes 1\otimes y_{i},
    r_{23}:=\sum_{i}1\otimes x_{i}\otimes y_{i}.
\end{equation}
Here $\G$ is embedded in its universal
enveloping algebra with unit $1$.
A solution of the CYBE is  called a classical
$r$-matrix (or simply an $r$-matrix).

Since the early stage of the study, it has been found necessary
to expand the tensor form  of the CYBE to operator
equations. After Semenov-Tian-Shansky~\cite{STS} showed that
skew-symmetric \rs in certain Lie algebras can be characterized as
linear operators satisfying a (Rota-Baxter) identity, the notion
of an $\mathcal O$-operator (also called a relative Rota-Baxter operator) of a Lie algebra was introduced by
Kupershmidt~\cite{Ku3} as a natural generalization of the CYBE,
and a skew-symmetric $r$-matrix is interpreted as an $\mathcal
O$-operator associated to the coadjoint representation. In the
opposite direction, it was shown in~\cite{Bai}
that any $\mathcal O$-operator gives an $r$-matrix  in a semi-direct product Lie algebra.
Furthermore,  pre-Lie algebras, in addition to their independent
interests in deformation theory, geometry, combinatorics, quantum field
theory and operads~\cite{Bu,Man}, naturally give rise to $\mathcal
O$-operators and hence provide \rs~\cite{Bai,TBGS}.
Independently, the Rota-Baxter identity appeared in the algebraic approach of Connes-Kreimer to renormalization of quantum field theory~\cite{CK} as a fundamental structure.

\vspace{-.2cm}

\subsection{The  \supcybe  and the imbalance between even and odd solutions}
The super classical Yang-Baxter equation (\supcybe) in its tensor form is given by Eq.~\eqref{eq:1}
except that the commutators are taken in the superspaces:
{\small
\begin{eqnarray}
   &&[r_{12},r_{13}]:=\sum_{i,j}(-1)^{|y_i||x_j|}[x_{i},x_j]\otimes y_{i}\otimes y_j,\notag\\
  &&[r_{12},r_{23}]:=\sum_{i,j}x_{i}\otimes [y_{i},x_j]\otimes y_j,\notag\\
\notag &&[r_{13},r_{23}]:=\sum_{i,j}(-1)^{|y_i||x_j|}x_{i}\otimes x_j\otimes [y_{i},y_j].
\end{eqnarray}
}
We call a solution of the \supcybe  a {\bf \sr}.

Since its appearance in the 1990s, the \supcybe  has
attracted considerable attention and has been applied widely. In \cite{ZGB}, the \supcybe  was studied by examining the classical limit of the quantum supergroup equations, and the possibility of using \srs to construct integrable supersymmetric models was also discussed.
The super version of the notion of Poisson manifold and its quantization were studied in \cite{An}, where the 2-tensors providing Poisson brackets were characterized  as those satisfying the \supcybe.
The work~\cite{GZB} revealed a close relationship between the \supcybe and Lie superbialgebras.
A universal quantization of Lie superbialgebras has been given in \cite{Geer} and it was  proved that
the Etingof-Kazhdan quantization is isomorphic to the Drinfeld-Jimbo type quantization for a simple Lie superalgebra of type $A$-$G$.
Lie superbialgebras,  as the underlying symmetry algebra, also play an important role in the integrable structure of AdS/CFT correspondence \cite{BS}.

However, there is a notable disparity between the even and odd
super $r$-matrices and a lack of connection between these two types of super $r$-matrices.
Most of the super
$r$-matrices obtained so far are even~\cite{An, BB, ER, GZB, Ka1}.
Furthermore, the
effective $\OO$-operator construction of $r$-matrices was expanded to constructing only even super $r$-matrices. In
particular, pre-Lie superalgebras were used to construct even
 $\OO$-operators and then even super
$r$-matrices~\cite{WHB}.
In contrast, little is known about the construction of odd super $r$-matrices.
It was not known how an $\OO$-operator could give an odd super $r$-matrix, nor were  there  natural selections of odd $\OO$-operators.

It is desirable to obtain a good understanding of the odd super $r$-matrices and their relationship with the even super $r$-matrices. In addition to their theoretical significance, odd super $r$-matrices have found applications in areas such as Kronecker elliptic functions and universal $R$-matrices~\cite{KT,LOZ}.
Moreover, while even  skew-supersymmetric super $r$-matrices induce Lie superbialgebras, it is
expected that odd  supersymmetric super $r$-matrices induce odd Lie superbialgebras introduced in \cite{V}, and a relation between the even and odd super $r$-matrices could shed light on relating the even and odd Lie superbialgebras.

\vspace{-.4cm}

\subsection{Parity duality for super $r$-matrices}
In this paper, we show that, the operator approach not only furnishes a broader context to study the \supcybe, but also provides an effective procedure to obtain both even and odd solutions by a parity duality of $\OO$-operators.

To begin with, by a unified treatment, we interpret the
homogeneous super $r$-matrices (that is, even and odd super
$r$-matrices) satisfying a certain supersymmetric condition
(called pan-supersymmetry) as the homogeneous $\mathcal
O$-operators with the same parity associated to the coadjoint
representations. Conversely, any homogeneous $\mathcal O$-operator
of a Lie superalgebra gives a homogeneous super $r$-matrix with
the same parity in a semi-direct product Lie superalgebra.

More significantly, the simultaneously developed yet seemingly unrelated even super $r$-matrices and odd super $r$-matrices turn out to be closely interconnected. In fact the connection is induced by a one-one correspondence between the even $\OO$-operators and the odd $\OO$-operators, called the {\bf parity duality}, made possible by the \parityrev representation of a
given representation of a Lie superalgebra.

The parity duality matches any homogeneous $\OO$-operator $T$ with another homogeneous $\OO$-operator $T^s$ with the opposite parity.
Then by the above uniform approach, a homogeneous $\OO$-operator
always gives a parity pair of super $r$-matrices $r_T$ and $r_{T^s}$ in enlarged Lie superalgebras by a semi-direct product
construction. The correspondence between $r_T$ and $r_{T^s}$ can
also be regarded as a parity duality between those super $r$-matrices coming from $\OO$-operators. For distinction, we call the
duality between the $\OO$-operators $T$ and $T^s$ the \textbf{$\OO$-operator
duality}, and the duality between their induced super $r$-matrices $r_T$ and $r_{T^s}$ the \textbf{$r$-matrix duality}. In summary we have
\begin{equation}
    \begin{split}
        \xymatrix{
            T \ar@{<->}[d] \ar@{<->}[rrr]^{\OO\text{-operator duality}} &&&T^s \ar@{<->}[d]   \\
            r_T  \ar@{<->}[rrr]^{r\text{-matrix duality}}   &&&r_{T^s}
        }
    \end{split}
    \label{eq:dual}
\end{equation}
It is worth stressing that the parity duality between super $r$-matrices holds only for those induced by $\OO$-operators, exhibiting a new advantage, distinctively for Lie superalgebras, of the approach to the \supcybe by $\mathcal O$-operators.

These dualities have several direct yet significant applications.
First, since a pre-Lie superalgebra naturally gives rise to an
even $\OO$-operator of the sub-adjacent Lie superalgebra, it also
gives an odd $\OO$-operator, and hence defines an odd (as well as an
even) super $r$-matrix.

Second, any pan-supersymmetric \sr , no matter even or odd,
 corresponds to an $\OO$-operator associated
to the coadjoint representation of the Lie superalgebra, which
gives rise to a pan-supersymmetric super $r$-matrix with the
same parity in an enlarged Lie superalgebra. By the $\OO$-operator
duality, there is also an $\OO$-operator with the opposite parity
and, as noted just above, gives rise to a pan-supersymmetric
super $r$-matrix with the latter parity in another enlarged Lie
superalgebra. Furthermore, by the same treatment, the two (even
and odd) pan-supersymmetric super $r$-matrices lead to two
more parity pairs of $\OO$-operators and then to two more parity
pairs of pan-supersymmetric super $r$-matrices in the next
level of enlarged Lie superalgebras. This process can be iterated,
eventually leading to a tree family of super $r$-matrices. See the diagram in Eq.~(\ref{di:tree}).

Finally, in the special case that the $\mathcal O$-operators are
associated to a self-reversing representation, that is, a
representation  isomorphic to its parity reversing
representation, the $\mathcal O$-operator gives a pair of
even and odd super $r$-matrices in
the \emph{same} enlarged Lie superalgebras by the semi-direct product
construction. This means that the $r$-matrix duality for such
super $r$-matrices holds in the same Lie superalgebra.

\vspace{-.2cm}

\subsection{Layout of the paper}

Here is an outline of the paper.

In Section~\ref{sec:cybe}, we  interpret a pan-supersymmetric \sr
as an $\mathcal O$-operator associated to the coadjoint
representation  (Theorem~\ref{thm:1}). As a consequence, there is
a correspondence between 2-cocycles on  Lie superalgebras  and
\srs in the non-degenerate and pan-supersymmetric case (Theorem~\ref{prop:3.9}).
Conversely, we give a construction of a pan-supersymmetric
\sr in a semi-direct product Lie superalgebra from an $\mathcal
O$-operator (Theorem~\ref{thm:2}).

\smallskip

In  Section~\ref{ss:dual}, we
first give the notion of the \parityrev of a representation
of a Lie superalgebra, obtained from the original representation
by a suspension process. We then obtain a natural one-one
correspondence between  $\OO$-operators with any given parity
associated to a representation and $\OO$-operators with the opposite
parity associated to the \parityrev representation
(Theorem~\ref{thm:sV}).
 Moreover,  we establish the
duality between even and odd \srs in semi-direct product Lie
superalgebras constructed from $\OO$-operators (Theorem~
\ref{cor:cons} and Corollary~\ref{co:rrs}). In
Section~\ref{ss:tree}, together with the procedure of getting  an
$\OO$-operator from a pan-supersymmetric \sr
in Section~\ref{sec:cybe},
 we obtain a replicating procedure that starts from a pan-supersymmetric \sr with any parity and arrives at a parity pair of pan-supersymmetric super $r$-matrices in semi-direct product Lie superalgebras. Repeating this procedure yields a whole family
of \srs, parameterized by the vertices of an infinite planar
binary tree (the diagram in Eq.~(\ref{di:tree})).
Section~\ref{ss:selfrev} focuses on self-reversing
representations. Any representation can be extended
to a self-reversing representation. For a self-reversing
representation, the duality is taken between the $\OO$-operators (resp. their derived \srs) for the same representations (resp. in the
same Lie superalgebras) (Corollary~\ref{cor:equ}).

\smallskip

In Section~\ref{sec:prelie}, we give a close relationship
between even or odd invertible $\mathcal O$-operators and pre-Lie
superalgebras (Theorem~\ref{prop:2.3}). In particular, the
identity map on a pre-Lie superalgebra induces an even $\mathcal
O$-operator of the sub-adjacent Lie superalgebra and hence, as a
direct consequence of the $\OO$-operator duality, an odd $\OO$-operator.
Subsequently, a single pre-Lie superalgebra gives rise to a parity
pair of super $r$-matrices in semi-direct product Lie superalgebras (Proposition~\ref{coro:4.5}).

\smallskip

\noindent {\bf Conventions:} Throughout this paper, take $\Z_2=\Z/(2)=\{\bar 0, \bar 1\}$. A vector superspace $V$ is  $\Z_2$-graded: $V=V_{\bar 0}\oplus V_{\bar 1}$. A homogeneous element $x\in V$ is called {\bf even} or {\bf odd} if its degree $|x|$ is zero or one.
All  vector superspaces and superalgebras are assumed to be finite-dimensional over an algebraically closed field $\mathbb F$ of characteristic zero.
The composition of maps $\phi$ and $\psi$ is denoted by $\phi\psi$.

\section{Super \rs and $\mathcal {O}$-operators of Lie superalgebras}
\label{sec:cybe} In this section, we establish the relationship between homogeneous super $r$-matrices and $\OO$-operators of a Lie superalgebra associated to a representation,
generalizing the connections for the $r$-matrices and even super $r$-matrices~\cite{Bai,Ku3,STS,WHB}. On the one hand, we prove that under the
pan-supersymmetric condition, a \sr is equivalent to an  $\mathcal O$-operator of a Lie superalgebra associated to the coadjoint representation. On the
other hand, starting from an $\mathcal O$-operator that is either
even or odd, we construct a \sr with the same parity in
a semi-direct product Lie superalgebra.

\subsection{$\OO$-operators of Lie superalgebras}
Let $\G_1$ and $\G_2$ be Lie superalgebras. A
homomorphism of Lie superalgebras is an even linear map $f:\G_1
\ra \G_2$ satisfying $f([x, y])=[f(x), f(y)]$ for all $x,y\in
\G_1$.
 A \textbf{representation} of a Lie
superalgebra $\G=\G_{\bar{0}}\oplus \G_{\bar{1}}$ in a vector
superspace $V=V_{\bar 0}\oplus V_{\bar 1}$ is a pair $(V, \rho)$
for a Lie superalgebra homomorphism $\rho:\G\longrightarrow
\mathfrak{gl}(V)$. We also say that $V$ is a \textbf{$\G$-module}.
Two representations $(V_1,\rho_1)$ and $(V_2,\rho_2)$ of a Lie
superalgebra $\G$ are called {\bf isomorphic} if there is an even
linear isomorphism $\phi: V_1\longrightarrow V_2$ satisfying
\begin{equation}
    \phi \rho_1(x)=\rho_2(x) \phi,\;\;\forall x\in \G.\label{eq:equivalence}
\end{equation}

\begin{defn} \label{defn:2.9}
Let ${\frak g}$ be a Lie superalgebra and $(V, \rho)$ be a representation of $\G$.  A homogeneous linear map $T: V \longrightarrow
{\frak g}$ is called an \textbf{$\mathcal {O}$-operator} of
${\frak g}$ associated to  $(V,\rho)$ if $T$ satisfies
\begin{equation} \label{eq:oop}
[T(v), T(w)] = T\Big((-1)^{(|T|+|v|)|T|}\rho(T(v))w-
(-1)^{|v|(|T|+|w|)}\rho(T(w))v\Big),\;\;\forall v, w \in V.
\end{equation}
Here by convention, the equation is required to be held for homogeneous elements and then extended to all elements by linearity. \emph{The same convention is applied to all equations for vector superspaces.}
\end{defn}

\begin{remark}
The choice of signs in Eq.~\eqref{eq:oop} is guided by the Koszul sign rule: whenever two homogeneous elements $u$ and $v$ are interchanged, the result is multiplied by $(-1)^{|u||v|}$. In this sense, the sign $(-1)^{(|T|+|v|)|T|}$ appears in front of the first term on the right-hand side of Eq.~\eqref{eq:oop} because the term can be obtained from the left-hand side by interchanging $T(v)$ and the map $T$ in $T(w)$. 
Likewise the sign $(-1)^{|v|(|T|+|w|)}$ appears in Eq.~\eqref{eq:oop}
from interchanging $v$ and $T(w)$. 
Note that, for the second term on the right-hand side of Eq.~\eqref{eq:oop}, we can alternatively first interchange $T(v)$ and $T(w)$ on the left-hand side, and then interchange $T(w)$ and the map $T$ in $T(v)$. This leads to a sign $(-1)^{|T(v)||T(w)|+|T||T(w)|}$, which agrees with the sign $(-1)^{|v|(|T|+|w|)}$ thanks to the identity $(-1)^{|T(v)||T(w)|+|T||T(w)|}=(-1)^{|v|(|T|+|w|)}.$
\end{remark}

An $\OO$-operator is called {\bf even} (resp. {\bf odd}) if it is an even (resp. odd) linear map.
Let $\OO_{\bar 0}(\G;V,\rho)$ and $\OO_{\bar 1}(\G;V,\rho)$ denote the sets of even  and odd $\OO$-operators of $\G$ associated to $(V,\rho)$ respectively.

As a special case, if $R$ is an $\mathcal {O}$-operator of  a Lie
superalgebra ${\frak g}$ associated to the adjoint representation,
then $R$ is called a \textbf{Rota-Baxter operator} of weight 0
on ${\frak g}$. Thus a Rota-Baxter
operator $R$ of weight 0 on $\G$ is defined by
$$[R(x), R(y)] = R\Big((-1)^{(|R|+|x|)|R|}[R(x),y]+[x,R(y)]\Big),\;\;\forall x, y \in \G.$$
In the literature, an $\OO$-operator is also called a relative Rota-Baxter operator. 

The following result is straightforward by transporting the structures.

\begin{lemma}\label{prop3.2}
Let $\phi: V_1\ra V_2$ be an isomorphism between two representations $(V_1,\rho_1)$ and $(V_2,\rho_2)$ of a Lie  superalgebra $\G$.
Then for each  $T\in \OO_{\alpha}(\G;V_2,\rho_2)$,  the composition  $T \phi$ is in $\OO_{\alpha}(\G;V_1,\rho_1)$, giving rise to a one-one correspondence between $\OO_{\alpha}(\G;V_1,\rho_1)$ and $\OO_{\alpha}(\G;V_2,\rho_2)$
for $\alpha\in \Z_2$.
\end{lemma}

We recall some  facts on  vector superspaces from \cite{CW,Sc}. Let  $V=V_{\bar 0}\oplus
V_{\bar 1}$ and $W=W_{\bar 0}\oplus W_{\bar 1}$ be two vector superspaces over a field $\F$.
A bilinear form $(\; , \;): V\times W \longrightarrow \mathbb F$
is called
 \textbf{odd}  if $(V_{\bar{0}}, W_{\bar{0}})=(V_{\bar{1}}, W_{\bar{1}})=0$, and
 \textbf{even}  if $(V_{\bar{0}}, W_{\bar{1}})=(V_{\bar{1}}, W_{\bar{0}})=0$;
\textbf{non-degenerate} if  $(v, w)=0$
for all $w\in W$ implies $v=0$, and $(v, w)=0$
for all $v\in V$ implies $w=0$.
The linear dual $V^*={\rm Hom}(V, \mathbb F)$  of $V$ inherits a
$\mathbb{Z}_2$-gradation $V^*=V^*_{\bar
0}\oplus V^*_{\bar 1}$ with
\begin{equation}\label{eq:2.3}
    V^*_\alpha:=\big\{u^*\in V^*| u^*(V_{\alpha+\bar
        1})=\{0\}\big\},\;\; \forall \alpha\in \Z_2.
\end{equation}
Given a linear form $u^*\in V^*$ and a vector $v\in V$, the resulting $u^*(v)$ is also denoted by $\langle u^*,v\rangle$.  It defines an even non-degenerate bilinear form
$\langle -,-\rangle:V^* \times V\longrightarrow \mathbb F$, which is called the canonical pairing.
The canonical pairings induces a pairing  on $V\times V^*$ by
\begin{equation}\label{eq:2.7}\langle v,
    u^*\rangle=(-1)^{|u^*||v|}\langle u^*,v\rangle,\;\;\forall  u^*\in V^*,v\in V.
\end{equation}
We  extend $\langle -,- \rangle$ to $(V \otimes V)^* \times (V \otimes V)$ by setting (applying $(V \otimes V)^* =V^*\otimes V^*$ since $V$ is finite dimensional)
$$\langle u_1^*\otimes u_2^*, v_1\otimes v_2\rangle:=(-1)^{|u_2^*||v_1|}\langle u_1^*,  v_1\rangle\langle u_2^*,  v_2\rangle, \quad \forall u_1^*, u_2^*\in V^*, v_1, v_2 \in V.$$

Let $(V,\rho)$ be a representation of a Lie
superalgebra $\G$. The even linear map $\rho^*:\G\longrightarrow
\mathfrak{gl}(V^*)$ defined by
\begin{equation}\label{eq:1.1}
    \langle\rho^*(x)u^*,v\rangle:=-(-1)^{|x||u^*|}\langle
    u^*,\rho(x)v\rangle,\;\;\forall x\in {\frak g}, u^*\in V^*, v\in V,
\end{equation}
gives  a representation $(V^*,\rho^*)$ of $\G$, called the {\bf dual representation} of $(V,\rho)$.
A representation $(V,\rho)$ of $\G$ is called {\bf self-dual} if it is isomorphic to  $(V^*,\rho^*)$.
The dual representation of the adjoint
representation  of a Lie superalgebra $\G$ is called the {\bf coadjoint representation} of $\G$.

\begin{exam}\label{exam:3.2}
Let $\G=\G_{\bar 0}\oplus\G_{\bar 1}$ be the $1|1$-dimensional Lie superalgebra with a homogeneous basis $\{e, f\}$, $e \in  \G_{\bar 0}$ and $f\in \G_{\bar 1}$, and non-zero  product $[e, f]=f$. 
Note that in particular, $[f,f]=0$.
Denote its dual vector superspace by $\G^*$ with the dual basis  $e^*\in \G^*_{\bar 0}, f^*\in \G^*_{\bar 1}$.
 We define an even linear map $T_0: \G^*\longrightarrow \G$ and an odd linear map $T_1: \G^*\longrightarrow \G$ by
    $$T_0(e^*)=0, \ T_0(f^*)=-f\;\;\textrm{ and }\;\; T_1(e^*)=f,\  T_1(f^*)=-e,$$ respectively. It is straightforward to check that $T_0$ (resp. $T_1$) is an even (resp. odd) $\mathcal{O}$-operator of $\G$ associated to the coadjoint representation.
\end{exam}

\begin{defn} (\cite{ABB,  Ch})
 Let $\G=\G_{\bar{0}}\oplus\G_{\bar{1}}$ be a Lie superalgebra and $x, y, z$ be homogeneous elements in $\G$.
 A bilinear form $\beta: \G\times \G \longrightarrow \mathbb F$  on $\G$ is said to be
\begin{enumerate}
\item \textbf{supersymmetric}  if $\beta(x,
y)=(-1)^{|x||y|}\beta(y, x)$;
\item \textbf{skew-supersymmetric}  if $\beta(x,
y)=-(-1)^{|x||y|}\beta(y, x)$;
\item \textbf{invariant}
if $\beta([x, y], z)=\beta(x, [y, z])$;
\item a \textbf{2-cocycle} on $\G$ if $\beta$ is
skew-supersymmetric and satisfies
\begin{equation}
\beta([x, y], z)=(-1)^{|y||z|}\beta([x, z], y)
+\beta(x,  [y, z]).
\end{equation}
\end{enumerate}
\end{defn}

As for Lie algebras, the adjoint representation $(\G, {\rm ad})$ of a Lie superalgebra $\G$ is self-dual  if and only if  there is an  even  non-degenerate invariant  bilinear form on $\G$~\cite{CW}. 
Then applying Lemma~\ref{prop3.2}, we obtain
\begin{coro} \label{coro:3.7}
Let $\G$ be a Lie superalgebra which admits an even non-degenerate invariant  bilinear form $\beta$ and let $\phi:\G\longrightarrow \G^*$ be an even linear map given by 
\begin{equation} \notag 
        \langle \phi(x), y\rangle:=\beta(x,y),\;\;\forall x, y \in \G.
    \end{equation}
    Then $T$ is in $\OO_{|T|}(\G;\G^*, {\rm ad}^*)$ if and only if the linear map
    $T\phi: \G\longrightarrow \G$ is a Rota-Baxter operator of weight 0 with degree $|T|$ on
    ${\frak g}$.
\end{coro}

\subsection{The interpretation of super $r$-matrices by $\mathcal O$-operators}
\label{ss:cybeo}
For a vector superspace $V$, under the natural linear isomorphism $V\otimes V\cong \Hom(V^*,V)$, every element $r=\sum_{i}x_{i}\otimes y_{i}\in V\otimes V$ corresponds to a unique linear map  $T_r:
V^*\longrightarrow V$ given by
\begin{equation}\label{eq:2.1}
\langle v^*, T_r(w^*)\rangle:=(-1)^{|r||w^*|}\langle v^*\otimes
w^*,r\rangle,\;\; \forall v^*,w^*\in V^*.
\end{equation}
We say that $r$ is {\bf even} (resp. {\bf odd}) if for all $i$, $x_i$ and $y_i$ have the same (resp. opposite) parity. Note that $r$ and $T_r$ have the same parity. The element $r$ is called \textbf{non-degenerate} if  $T_r$ is
invertible.

The \textbf{twist map} is an even linear map $\sigma: V\otimes V \longrightarrow V\otimes V$ defined by
$$\sigma(v \otimes w) :=(-1)^{|v||w|} w \otimes v, \quad \forall v, w \in
 V.$$
For  $r\in V\otimes V$,  $r$ is called \textbf{supersymmetric} (resp. {\bf skew-supersymmetric}) if $\sigma(r)=r$ (resp. $\sigma (r)=-r$), and is called
 {\bf \pansym} if it is either odd supersymmetric or even skew-supersymmetric.
Then a \pansym element $r$ is characterized by $\sigma
(r)=-(-1)^{|r|}r$.

We present some general facts on vector superspaces.
\begin{lemma}\label{lem2.6}
Let $V$ be a vector superspace and  $r=\sum_{k}x_{k}\otimes
y_{k}\in V\otimes V$ be a homogeneous element. Then for all
$v^* \in V^*$, we have $T_r(v^*)=(-1)^{|v^*|}\sum_k\langle
v^*,y_k\rangle x_k$. In particular, if $r$ is \pansym, then
$T_r(v^*)=-(-1)^{|r|+|r||v^*|}\sum_k\langle
v^*,x_k\rangle y_k$.
\end{lemma}

\begin{proof}
Note that  $|r|=|x_k|+|y_k|$. For all $v^*, w^*\in V^*$, we have
\begin{equation*}
\langle w^*,T_r(v^*)\rangle=(-1)^{|r||v^*|}\langle w^*\otimes v^*,
r\rangle= \sum_{k}(-1)^{|y_k||v^*|}\langle
w^*,x_{k}\rangle\langle v^*,y_{k}\rangle=\Big\langle
w^*,(-1)^{|v^*|}\sum_k\langle v^*,y_k\rangle x_k\Big\rangle.
\end{equation*}
 We have $|y_k|=|v^*|$ and
$T_r(v^*)=(-1)^{|v^*|}\sum_k\langle v^*,y_k\rangle x_k$, since the canonical pairing is even and non-degenerate. For
the second statement,  we see that
\begin{eqnarray*}
\langle w^*,T_r(v^*)\rangle &=&(-1)^{|r||v^*|}\langle w^*\otimes v^*, r\rangle
=(-1)^{|r||v^*|}\langle w^*\otimes v^*, -(-1)^{|r|}\sigma (r)\rangle\\
&=&-(-1)^{|r|+|r||v^*|}\sum_{k}\langle w^*,y_{k}\rangle\langle v^*,x_{k}\rangle
=\Big\langle w^*,-(-1)^{|r|+|r||v^*|}\sum_k\langle v^*,x_k\rangle y_k\Big\rangle.
\end{eqnarray*}
Then we have $T_r(v^*)=-(-1)^{|r|+|r||v^*|}\sum_k\langle v^*,x_k\rangle y_k$.
\end{proof}

\begin{lemma}\label{lem2.7}
Let $V$ be a vector superspace and $r=\sum_{k}x_{k}\otimes y_{k}\in V\otimes V$  be a homogeneous element. Then $r$ is \pansym
if and only if $\langle w^*,T_r(v^*)\rangle=-(-1)^{|v^*||w^*|}\langle v^*, T_r(w^*)\rangle$ for all $w^*, v^*\in V^*$.
\end{lemma}
\begin{proof}
By the evenness of the canonical pairing and  Lemma \ref{lem2.6},  for all $w^*,v^*\in V^*$, we have
\begin{eqnarray*}
&&\langle w^*\otimes v^*, \sigma(r)+(-1)^{|r|}r\rangle\\
&=&\Big\langle w^*\otimes v^*, \sum_{k}(-1)^{|x_k||y_k|}y_{k}\otimes x_{k}\Big\rangle+(-1)^{|r|+|r||v^*|} \langle w^*, T_r(v^*)\rangle\\
&=&\Big\langle v^*, \sum_{k}\langle  w^*, y_k\rangle x_k\Big\rangle+(-1)^{|r|+|r||v^*|} \langle w^*, T_r(v^*)\rangle\\
&=&\langle v^*, (-1)^{ |w^*|}T_r(w^*)\rangle+(-1)^{|r|+|r||v^*|} \langle w^*, T_r(v^*)\rangle \\
&=&(-1)^{|r|+|r||v^*|}\big((-1)^{|v^*||w^*|}\langle v^*, T_r(w^*)\rangle+ \langle w^*,T_r(v^*)\rangle\big).
\end{eqnarray*}
Since the canonical pairing is
non-degenerate, we obtain the desired result.
\end{proof}

Let $\G$ be a Lie superalgebra and
$r=\sum_{i}x_{i}\otimes y_{i}\in \G\otimes \G$.
Then $r$ is called a {\bf \sr} in $\G$ if it is a solution of the \supcybe
\begin{equation}\label{eq:3.1}
    [[r,r]]:=[r_{12}, r_{13}]+[r_{12}, r_{23}]+[r_{13}, r_{23}]=0,
\end{equation}
where $\G$ is embedded in  the universal enveloping superalgebra $U(\G)$ and
\begin{equation} \label{eq:18}
    r_{12}:=\sum_{i}x_{i}\otimes y_{i}\otimes 1,
    r_{13}:=\sum_{i}x_{i}\otimes 1\otimes y_{i},
    r_{23}:=\sum_{i}1\otimes x_{i}\otimes y_{i}.
\end{equation}
As noted in the Introduction, the commutators
in Eq. \eqref{eq:3.1} are in the superspaces in the sense that
\begin{eqnarray}
    &&[r_{12},r_{13}]=\sum_{i,j}(-1)^{|y_i||x_j|}[x_{i},x_j]\otimes y_{i}\otimes y_j,\label{eq:1.3}\\
    &&[r_{12},r_{23}]=\sum_{i,j}x_{i}\otimes [y_{i},x_j]\otimes y_j,\label{eq:1.4}\\
   \label{eq:1.5} &&[r_{13},r_{23}]=\sum_{i,j}(-1)^{|y_i||x_j|}x_{i}\otimes x_j\otimes [y_{i},y_j].
\end{eqnarray}
We denote by $\Sol_{\bar 0}(\G)$ and $\Sol_{\bar 1}(\G)$ the sets of even and odd super $r$-matrices in $\G$ respectively.

We now establish a close relationship between homogeneous
pan-supersymmetric super $r$-matrices and  $\OO$-operators
associated to the coadjoint representations,
generalizing the work of Kupershmidt \cite{Ku3} for $r$-matrices
and \cite[Proposition~1]{WHB} for even super $r$-matrices.

\begin{theorem}\label{thm:1}
Let $\G$ be a Lie superalgebra and $r\in\G\otimes\G$ be \pansym. Then $r\in \Sol_{|r|}(\G)$ if and only if
$T_r\in\OO_{|r|}(\G;\G^*,\ad^*)$.
\end{theorem}

\begin{proof}
Let $a^*, b^*, c^*\in \G^*$ and  $r=\sum_{k}x_{k}\otimes y_{k}\in \G\otimes \G$ be homogeneous and \pansym.
By Lemma \ref{lem2.6}, we have
\begin{eqnarray*}
&&\lefteqn{\langle a^*\otimes b^*\otimes c^*, [r_{12},r_{13}]\rangle =\sum_{k, l}(-1)^{|y_k||x_l|}\langle a^*\otimes b^*\otimes c^*, [x_{k},x_l]\otimes y_{k}\otimes y_l\rangle}\\
&=& \sum_{k, l}(-1)^{|y_k||x_l|+(|x_k|+|x_l|)(|b^*|+|c^*|)+|y_k||c^*|} \langle a^*,
[x_{k},x_l]\rangle \langle b^*, y_{k}\rangle \langle c^*,
y_l\rangle\\
&=&\sum_{k,l}(-1)^{|b^*||r|} \langle a^*,
[(-1)^{|b^*|}\langle b^*, y_{k}\rangle x_{k}, (-1)^{|c^*|}\langle c^*,
y_l\rangle x_l]\rangle \\
&=& (-1)^{|b^*||r|} \langle a^*,
[T_r(b^*),  T_r(c^*)]\rangle,
\end{eqnarray*}
here $\langle b^*, y_{k}\rangle=\langle c^*,
y_l\rangle=0$ unless $|b^*|=|y_k|$ and $|c^*|=|y_l|$ since the canonical pairing is even. Similarly,  we have
\begin{eqnarray*}
&&\langle a^*\otimes b^*\otimes c^*, [r_{12},r_{23}]\rangle = -(-1)^{|r|+|a^*||r|+|a^*||b^*|} \langle b^*,
[T_r(a^*),  T_r(c^*)]\rangle,\\
&&\langle a^*\otimes b^*\otimes c^*, [r_{13},r_{23}]\rangle = (-1)^{|a^*||r|+|a^*||c^*|+|b^*||c^*|} \langle c^*,
[T_r(a^*),  T_r(b^*)]\rangle.
\end{eqnarray*}

On the other hand, it follows from Lemma~\ref{lem2.7} that
\begin{eqnarray*}
\langle a^*, -(-1)^{(|r|+|b^*|)|r|}T_r(\ad^*(T_r(b^*))c^*)\rangle
= (-1)^{|a^*||r|+|b^*||r|+|a^*||c^*|+|b^*||c^*|} \langle c^*,
[T_r(a^*),  T_r(b^*)]\rangle,
\end{eqnarray*}
\begin{eqnarray*}
\langle a^*, (-1)^{|b^*|(|r|+|c^*|)}T_r(\ad^*(T_r(c^*))b^*)\rangle
= -(-1)^{|a^*||r|+|b^*||r|+|a^*||b^*|+|r|} \langle b^*,
[T_r(a^*),  T_r(c^*)]\rangle.
\end{eqnarray*}
Then
\begin{eqnarray*}
&&\langle a^*, [T_r(b^*), T_r(c^*)] - (-1)^{(|r|+|b^*|)|r|}T_r(\ad^*(T_r(b^*))c^*)+
(-1)^{|b^*|(|r|+|c^*|)}T_r(\ad^*(T_r(c^*))b^*)\rangle\\
&=&\langle a^*, [T_r(b^*), T_r(c^*)]\rangle-(-1)^{|a^*||r|+|b^*||r|+|a^*||b^*|+|r|}
\langle b^*,
[T_r(a^*),  T_r(c^*)]\rangle \label{eq:3.4}\\
&&+(-1)^{|a^*||r|+|b^*||r|+|a^*||c^*|+|b^*||c^*|}  \langle c^*,
[T_r(a^*), T_r(b^*)]\rangle\\
 &=&(-1)^{|b^*||r|}\langle a^*\otimes b^*\otimes c^*, [[r,r]]\rangle.
\end{eqnarray*}
Hence $r$
is a homogeneous super $r$-matrix in
$\G$ if and only if  $T_r$  is  an  $\mathcal {O}$-operator with
the same parity as $r$ associated to $(\G^*, \ad^*)$.
 \end{proof}
\begin{exam}\label{exam:4.4}
Consider the Lie superalgebra $\G$ and
$\mathcal{O}$-operators $T_0$ and $T_1$ of $\G$ associated to the
coadjoint representation in Example ~\ref{exam:3.2}. Let $r_0$ and $r_1$ be the corresponding tensors from $T_0$ and $T_1$ according to Eq.~\eqref{eq:2.1} respectively. Then
$$r_{0}= f\otimes
f\;\;\textrm{ and }\;\; r_{1}= e\otimes f+f\otimes e.$$
Obviously, $r_0$ is even and $\sigma (r_0)=-r_0$, whereas $r_1$ is
odd and $\sigma (r_1)=r_1$.
Then by Theorem ~\ref{thm:1},  $r_0$ and
$r_1$ are super $r$-matrices in the Lie superalgebra $\G$.
\end{exam}

Let $V$ be a  vector superspace. A non-degenerate element $r\in V\otimes V$ defines a bilinear form $\beta_r$ on $V$ by
\begin{equation}\label{eq:nonde}
\beta_r(u,v):=\langle T^{-1}_r(u), v\rangle , \;\;\forall u,v \in
V.
\end{equation}

\begin{lemma}\label{lem:r-b}
Let $V$ be a  vector superspace and  $r\in V\otimes V$ be a homogeneous
 non-degenerate element. Then $r$ is \pansym
if and only if $\beta_r(u, v)=-(-1)^{|u||v|}\beta_r(v, u)$ for all $u,v\in
V$.
\end{lemma}

\begin{proof}
Since $r$ is non-degenerate, for all $u, v\in V$, there exist
$a^*, b^*\in V^*$ such that $u=T_r(a^*), v=T_r(b^*),$ where $|u|=|a^*|+|r|, |v|=|b^*|+|r|$. If $\sigma(r)=-(-1)^{|r|} r$,  then by Lemma \ref{lem2.7}  we have
\begin{eqnarray*}
\beta_r(u,v)&=&\langle T_r^{-1}(u), v\rangle=\langle a^*, T_r(b^*)\rangle=-(-1)^{|r|+|r||a^*|}\langle T_r(a^*), b^*\rangle\nonumber\\
&=&-(-1)^{(|a^*|+|r|)(|b^*|+|r|)}\langle T_r^{-1}(v),
u\rangle=-(-1)^{|u||v|}\beta_r(v,u).\label{eq:3.8}
\end{eqnarray*}
It is straightforward to see that the converse also holds.
\end{proof}

Now we give a unified interpretation of both  even and odd
non-degenerate super $r$-matrices,
generalizing the well-known work of Drinfeld~\cite{Dr} and the even case in \cite[Corollary 2]{WHB}.

\begin{theorem}\label{prop:3.9}
Let $\G$ be a Lie superalgebra and $r\in\G\otimes\G$
be homogeneous, non-degenerate and \pansym.
Then $r\in\Sol_{|r|}(\G)$ if and only if $\beta_r$ is a 2-cocycle
on $\G$.
\end{theorem}

\begin{proof}
It follows from Lemma~\ref{lem:r-b} that $\beta_r$ is  skew-supersymmetric  when $r$ is \pansym. For all $x, y, z\in \G$, there
exist $a^*, b^*,c^*\in \G^*$ such that $x=T_r(a^*), y=T_r(b^*),
z=T_r(c^*),$ with $|x|=|a^*|+|r|, |y|=|b^*|+|r|,
|z|=|c^*|+|r|$. By Theorem~ \ref{thm:1}, $r\in\Sol_{|r|}(\G)$ leads to $T_r\in\OO_{|r|}(\G, \G^*,\ad^*)$ and we have
\begin{eqnarray*}
\beta_r([x, y], z) 
&=&\langle T_r^{-1}([T_r(a^*), T_r(b^*)]), z\rangle\nonumber\\
&=&\langle T_r^{-1}(T_r((-1)^{(|a^*|+|r|)|r|}\ad^*(T_r(a^*))b^*-
(-1)^{|a^*|(|b^*|+|r|)}\ad^*(T_r(b^*))a^*)), z\rangle\nonumber\\
&=& -(-1)^{(|a^*|+|r|)(|b^*|+|r|)} \langle b^*, \ad(T_r(a^*))z\rangle
+\langle a^*,  \ad(T_r(b^*)) z\rangle\nonumber\\
&=& -(-1)^{|x||y|}\langle T_r^{-1}(y), [x,z]\rangle
+\langle T_r^{-1}(x),  [y, z]\rangle \label{eq:3.9}\\
&=&-(-1)^{|x||y|}\beta_r(y, [x,z])+\beta_r(x,  [y, z])\nonumber\\
&=&(-1)^{|y||z|}\beta_r([x,z], y)+\beta_r(x,  [y, z])\nonumber.
\end{eqnarray*}
Hence $\beta_r$ is a  2-cocycle on $\G$. Conversely, if $\beta_r$
is a 2-cocycle on $\G$, then  $T_r$ is an    $\mathcal {O}$-operator of $\G$ associated to
$(\G^*,\ad^*)$ and thus  $r$ is a  super $r$-matrix  by Theorem~ \ref{thm:1}.
\end{proof}

\subsection{Super $r$-matrices from $\mathcal O$-operators}
We give the notion of the semi-direct product of a Lie
superalgebra and its module.

\begin{defn}
    Let $(V, \rho)$ be a representation of a Lie superalgebra $\G$. Define a Lie superalgebra structure on
    the direct sum ${\frak g}\oplus V$ of the underlying
    vector superspaces of ${\frak g}$ and $V$  by
\begin{equation}
    [(x,u),(y,v)]\coloneqq ([x,y], \rho(x)v-(-1)^{|u||y|}\rho(y)u),\;\;\forall x,y \in {\frak g}, u, v\in
        V.\end{equation}
  This  Lie superalgebra is called the
    \textbf{semi-direct product} of ${\frak g}$ and $V$, and is denoted by
    ${\frak g}\ltimes_\rho V$.
\end{defn}

We  evidently have
\begin{lemma}\label{prop:2.4}
    If $(V_1, \rho_1)$ and $(V_2, \rho_2)$ are isomorphic representations of a Lie superalgebra $\G$, then  their corresponding semi-direct product Lie  superalgebras
    ${\frak g}\ltimes_{\rho_1} V_1$ and  ${\frak g}\ltimes_{\rho_2} V_2$ are isomorphic.
\end{lemma}

Let $\{v_i\}$ be a homogeneous basis of a vector superspace $V$   and $\{v_i^*\}$ be the dual basis of $V^*$.
Since $\mathrm{Hom}(V,\G)\cong\G\otimes V^*$ which is embedded into $(\G\oplus V^*)\otimes (\G\oplus V^*)$,  we can identify an element $T\in\Hom(V, \G)$ with the $2$-tensor
\begin{equation}\label{eq3.1}
\mT:=\sum_{i} T(v_i)\otimes v_i^*\in (\G\oplus V^*)\otimes (\G\oplus V^*).
\end{equation}
Then $|\mT|=|T|$ and we  define
\begin{equation}\label{eq3.2}
r \coloneqq r_T\coloneqq \mT-(-1)^{|\mT|}\sigma(\mT)=\sum_{i} (T(v_i)\otimes v_i^*+(-1)^{(|T|+\bar 1)(|v_i|+\bar 1)} v_i^*\otimes T(v_i)).
\end{equation}
Obviously, $r$ is \pansym.

The following conclusion shows that any $\OO$-operator gives a super $r$-matrix. 
Note that in the case that $T$ is even (hence $r$ is even skew-supersymmetric) is already obtained in \cite[Proposition~2]{WHB}.

\begin{theorem}\label{thm:2} Let $(V, \rho)$ be a representation of a Lie superalgebra
$\G$ and $T: V\rightarrow \G$ be a homogeneous linear map.  Then the element
$r:=\mT-(-1)^{|\mT|}\sigma(\mT)$ is a \sr in the
semi-direct product Lie superalgebra $\G \ltimes_{\rho^*}V^*$  if
and only if $T\in \OO_{|T|}(\G;V,\rho)$.
\end{theorem}

\begin{proof}
To obtain an equivalent condition for $r$  being a \sr in $\G\ltimes_{\rho^*}V^*$, we first note that
\begin{eqnarray*}
&&\sum_{i,j} (-1)^{(|v_i|+|T|)|v_j|+|T|+|v_j|}\rho^*(T(v_i))v_j^*\otimes v_i^*\otimes T(v_j)\\
&=& \sum_{i,j} \Big(-(-1)^{|T|+|v_j|}\sum_{k} \langle v_j^*, \rho(T(v_i))v_k\rangle v_k^*\Big) \otimes v_i^*\otimes T(v_j)\\
&=&\sum_{i,k} -(-1)^{|v_i|+|v_k|}v_k^* \otimes v_i^*\otimes
T\Big(\sum_{j} \langle v_j^*, \rho(T(v_i))v_k\rangle v_j\Big)\\
&=&\sum_{i,j} -(-1)^{|v_i|+|v_j|}v_j^*\otimes v_i^* \otimes
T(\rho(T(v_i))v_j).
\end{eqnarray*}
Here the second equality follows from taking $v_j$
such that $|v_j|=|T|+|v_i|+|v_k|$, since  for all the other choices $\langle v_j^*,
\rho(T(v_i))v_k\rangle=0$. 
Similarly, we have
$$\sum_{i,j}(-1)^{|T|(|v_i|+|v_j|)+|v_i|}\rho^*(T(v_j))v_i^*\otimes T(v_i)\otimes v_j^*=\sum_{i,j}-(-1)^{|v_i||v_j|+|v_i|+|T|} v_i^*\otimes T(\rho(T(v_j))v_i)\otimes v_j^*.$$
Then we have
\begin{eqnarray*}
[r_{12}, r_{13}]
&=&\sum_{i,j}\Big((-1)^{(|v_j|+|T|)|v_i|}[T(v_i), T(v_j)]\otimes v_i^*\otimes v_j^*+(-1)^{|T|(|v_i|+|v_j|)+|v_i|}\rho^*(T(v_j))v_i^*\otimes T(v_i)\otimes v_j^*\\
&&\quad \;\;\;\;-(-1)^{(|v_i|+|T|)|v_j|+|T|+|v_j|}\rho^*(T(v_i))v_j^*\otimes v_i^*\otimes T(v_j)\Big)\\
&=&\sum_{i,j}\Big((-1)^{(|v_j|+|T|)|v_i|}[T(v_i), T(v_j)]\otimes v_i^*\otimes v_j^*
-(-1)^{|v_i||v_j|+|v_i|+|T|} v_i^*\otimes T(\rho(T(v_j))v_i)\otimes v_j^*\\
&&\quad \;\;\;\;
+(-1)^{|v_i|+|v_j|} v_j^*\otimes v_i^*\otimes T(\rho(T(v_i))v_j)\Big).
\end{eqnarray*}

Also apply the same computations to $[r_{12}, r_{23}]$ and $[r_{13}, r_{23}]$. Set
$$Op:=[T(v_i), T(v_j)] -(-1)^{(|T|+|v_i|)|T|}T(\rho(T(v_i))v_j)+ (-1)^{|v_i|(|T|+|v_j|)}T(\rho(T(v_j))v_i)$$
to be from the $\OO$-operator identity in Eq.~\eqref{eq:oop}. Then we  derive
\begin{eqnarray*}
&&[r_{12}, r_{13}]+[r_{12}, r_{23}]+[r_{13}, r_{23}]\\
&=&\sum_{i,j}\Big((-1)^{(|v_j|+|T|)|v_i|}[T(v_i), T(v_j)]\otimes v_i^*\otimes v_j^*
+(-1)^{|v_i|+|v_j|} v_i^*\otimes v_j^*\otimes T(\rho(T(v_j))v_i)\\
&&\quad\;\;\;\;\;-(-1)^{|v_i||v_j|+|v_i|+|T|} v_i^*\otimes T(\rho(T(v_j))v_i)\otimes v_j^*\\
&&\quad\;\;\;\;\;-(-1)^{|T|+|v_i|+|v_j|+|v_i||v_j|}v_i^*\otimes v_j^*\otimes T(\rho(T(v_i))v_j)
+T(\rho(T(v_j))v_i)\otimes v_i^*\otimes  v_j^*\\
&&\quad\;\;\;\;\;+(-1)^{|v_i|}v_i^*\otimes T(\rho(T(v_i))v_j) \otimes v_j^*+(-1)^{(|v_j|+|T|)|v_i|+|v_i|+|v_j|} v_i^*\otimes v_j^* \otimes[T(v_i), T(v_j)] \\
&&\quad\;\;\;\;\;-(-1)^{|T|+|v_i||v_j|}T(\rho(T(v_i))v_j)\otimes v_i^*\otimes v_j^*-(-1)^{|T||v_i|+|T|+|v_i|}v_i^*\otimes[T(v_i), T(v_j)]\otimes  v_j^*\Big)\\
&=&\sum_{i,j} (-1)^{(|v_j|+|T|)|v_i|}\Big(
Op\otimes v_i^*\otimes v_j^*
 -(-1)^{|T|+|v_i|+|v_j||v_i|}v_i^*\otimes Op\otimes v_j^* +(-1)^{|v_i|+|v_j|}v_i^*\otimes v_j^*\otimes Op\Big).
\end{eqnarray*}
Therefore, $r$ is a \sr if and only if $Op=0$, that is, $T$
is an  $\mathcal {O}$-operator of $\G$ associated to  $(V, \rho)$.
\end{proof}

Combining Theorems~ \ref{thm:2} and \ref{thm:1}, we obtain

\begin{coro}\label{cor:equv}
Let $(V,\rho)$ be a representation of a Lie superalgebra $\G$ and
$T:V\longrightarrow \G$ be  a homogeneous linear map. Then the
following statements are equivalent:
\begin{enumerate}
\item $T$ is in $\OO_{|T|}(\G;V, \rho)$; 
\item $r_T$ given by $T$ via Eq.~\eqref{eq3.2} is in $\Sol_{|T|}(\G
\ltimes_{\rho^*}V^*)$; 
\item $T_{r_T}$ given by
$r_T$  via Eq.~ \eqref{eq:2.1} is in $\OO_{|T|}(\G
\ltimes_{\rho^*}V^*; (\G \ltimes_{\rho^*}V^*)^*,
\ad^*)$.
\end{enumerate}
\end{coro}

\section{Parity dualities of $\OO$-operators and \srs}
\label{sec:oop}
\label{sec:parity}

Using the notion of the \parityrev of a Lie
superalgebra representation, we obtain a natural one-one correspondence between  $\OO$-operators of any given parity associated to a
representation and $\OO$-operators of the opposite parity
associated to the \parityrev representation. This correspondence
can be regarded as a duality between  even and odd $\OO$-operators
and  we apply it to establish a duality between even and odd \srs
in semi-direct product Lie superalgebras constructed from
$\OO$-operators. Moreover, this process is iterated to generate a
tree family of super $r$-matrices from any given homogeneous
pan-supersymmetric super $r$-matrix or
$\OO$-operator. Finally, this approach is carried out when the representation is self-reversing.

\subsection{Parity reverse representations and parity dualities}\label{ss:dual}

For a vector superspace $V=V_{\bar 0}\oplus V_{\bar 1}$, we denote by $sV$  the vector superspace obtained by interchanging the even and  odd parts of  $V$, that is, $(sV)_{\bar 0}=V_{\bar 1}$ and $(sV)_{\bar 1}=V_{\bar 0}$.
The {\bf \parityrev map} (or the \textbf{suspension operator}) is an odd linear map $s:V\ra sV$ that sends each homogeneous element of $V$ to the same element in $sV$ but with the opposite parity, that is, $s$ sends $v\in V_{\alpha}$ to $sv\in
(sV)_{\alpha+\bar 1}$ for all $\alpha\in\Z_2$.

Motivated by the constructions in~\cite{Sc,SZ}, we give the following notion.

\begin{prop-def}
Let $(V,\rho)$ be a representation of a Lie
    superalgebra $\G$. Define $\rho^s:\G\ra \mathfrak{gl}(sV)$ by
    \begin{equation}
        \rho^s(x)sv:=(-1)^{|x|}s(\rho(x)v),\;\;\forall x\in \G, sv\in sV.
        \label{eq:4.8}
    \end{equation}
    Then $(sV,\rho^s)$ is a representation  of $\G$, called the \textbf{\parityrev} of the representation $(V,\rho)$ of $\G$.
\end{prop-def}

We provide the following duality between even and odd $\OO$-operators.

\begin{theorem}\label{thm:sV}
Let $\G$ be a Lie superalgebra. Suppose that $(V,\rho)$ is a
representation of $\G$ and  $(sV,\rho^s)$ is the \parityrev
representation of $\G$. 
Then there exists a one-one correspondence between
$\OO_{\alpha}(\G;V,\rho)$ and $\OO_{\alpha+\bar 1}(\G; sV,\rho^s)$ for $\alpha\in \Z_2$.
\end{theorem}
This correspondence between $\OO_{\alpha}(\G;V,\rho)$ and $\OO_{\alpha+\bar 1} (\G;sV,\rho^s)$  for $\alpha\in \Z_2$ is called the {\bf $\OO$-operator duality}.
\begin{proof}
Let $T:V\ra \G$ be a homogeneous linear map.
Define another homogeneous linear map
\begin{equation}\label{eq:3.10}
T^s:  sV\longrightarrow \G  \; \textrm{ by }  \;T^s(su):=T(u), \;\;\forall su\in sV.
\end{equation}
Note that $s$ is an odd map, giving $|T^s|=|T|+\bar 1$.
For all $z\in \G$ and $su\in sV$, we have
$$T^s(\rho^s(z)su)=T^s((-1)^{|z|}s(\rho(z)u))=(-1)^{|z|}T(\rho (z)u).$$

Now let $T\in \OO_{\alpha}(\G;V, \rho)$. To verify that $T^s$ is
in $\OO_{\alpha+\bar 1}(\G; sV, \rho^s)$, for all
$sv$ and $sw$ in $sV$, we derive
\begin{eqnarray}\label{eq:Ts}
&&[T^s(sv), T^s(sw)]
=[T(v), T(w)]\nonumber\\
&=&T((-1)^{(|T|+|v|)|T|}\rho(T(v))w-(-1)^{|v|(|T|+|w|)}\rho(T(w))v)\nonumber\\
&=&T^s((-1)^{(|T|+|v|)|T|+(|T|+|v|)}\rho^s(T(v))sw-(-1)^{|v|(|T|+|w|)+(|T|+|w|)}\rho^s(T(w))sv)\\
&=&T^s((-1)^{(|T^s|+|sv|)|T^s|}\rho^s(T^s(sv))sw-(-1)^{|sv|(|T^s|+|sw|)}\rho^s(T^s(sw))sv),\nonumber
\end{eqnarray}
as needed.

Conversely,  suppose that $T^s$ is in $\OO_{\alpha+\bar 1}(\G; sV, \rho^s)$. Then  the linear map $T=(T^s)^s:V\ra \G$ given by Eq.~\eqref{eq:3.10} has degree $\alpha$ and satisfies $$[T(v), T(w)]=T((-1)^{(|T|+|v|)|T|}\rho(T(v))w-(-1)^{|v|(|T|+|w|)}\rho(T(w))v), \;\;\forall v, w \in V,$$
by the same computation as for Eq. ~\eqref{eq:Ts}. Thus $T$ is in $\OO_{\alpha}(\G;V, \rho)$.
\end{proof}

\begin{remark}\label{rmk:3.12}
One naturally asks whether the duality between even and odd
$\mathcal O$-operators in Theorem \ref{thm:sV} can be used to
give a similar duality between  even  and odd super $r$-matrices in a Lie superalgebra $\G$ by
Theorem ~\ref{thm:1}. Unfortunately,  this is not true in
general. In fact, suppose that $r\in \Sol_{|r|}(\G)$ is \pansym. Then
$T_r\in\OO_{|r|}(\G;\G^*,\ad^*)$ and we have another $\mathcal
O$-operator $T_r^s\in\OO_{|r|+\bar 1}(\G;s\G^*,(\ad^*)^s)$.
However, $T_r^s$ is in general no longer an $\mathcal O$-operator
associated to the coadjoint representation
$(\G^*, \ad^*)$ of $\G$ and thus it cannot give a super
$r$-matrix in $\G$.

Similarly, for a Rota-Baxter operator $R:\G\longrightarrow\G$ of weight zero. It is an $\mathcal O$-operator associated to the adjoint representation $(\G, {\rm ad})$. Then $R^s:s\G\longrightarrow \G$ is  an $\mathcal O$-operator associated to the \parityrev representation $(s\G,{\rm ad}^s)$. However, $R^s$ is not a Rota-Baxter operator of weight zero on $\G$ in general.
\end{remark}

For an $\OO$-operator with a given parity, Theorem~ \ref{thm:sV}
gives another $\OO$-operator with the opposite parity. Applying
Theorem~\ref{thm:2}, the pair gives a pair of even and odd
super $r$-matrices in (different)
semi-direct product Lie superalgebras. This phenomenon
demonstrates the advantage of the operator approach to the
\supcybe.

\begin{theorem}\label{cor:cons}
Let $(V,\rho)$ be a representation of a Lie superalgebra $\G$ and $T:V\longrightarrow \G$ be  a homogeneous linear map.
Then the following statements are equivalent:
\begin{enumerate}
\item $T$ is in $\OO_{|T|}(\G;V, \rho)$; \label{it:const}
\item $T^s$ is in $\OO_{|T|+\bar 1}(\G;sV, \rho^s)$; \label{it:consts}
\item $r_T$ given by $T$ via Eq.~\eqref{eq3.2} is in $\Sol_{|T|}(\G \ltimes_{\rho^*}V^*)$; \label{it:consr}
\item $r_{T^s}$ given by $T^s$ via Eq.~\eqref{eq3.2} is in $\Sol_{|T|+\bar 1}(\G \ltimes_{{(\rho^s)}^*}(sV)^*)$; \label{it:consrs}
\item $T_{r_T}$ given by $r_T$  via Eq.~ \eqref{eq:2.1} is in $\OO_{|T|}(\G \ltimes_{\rho^*}V^*; (\G \ltimes_{\rho^*}V^*)^*, \ad^*)$;\label{it:constr}
\item \label{it:constrs}
$T_{r_{T^s}}$ given by  $r_{T^s}$  via Eq.~ \eqref{eq:2.1} is in $\OO_{|T|+\bar 1}(\G \ltimes_{{(\rho^s)}^*}(sV)^*; (\G \ltimes_{{(\rho^s)}^*}(sV)^*)^*,\ad^*)$.
\end{enumerate}
\end{theorem}

\begin{proof}
 By Corollary~\ref{cor:equv}, we have $\eqref{it:const} \iff \eqref{it:consr} \iff  \eqref{it:constr}$ and, at the same time, $\eqref{it:consts} \iff \eqref{it:consrs} \iff \eqref{it:constrs}$.
It follows from Theorem~\ref{thm:sV} that $\eqref{it:const} \iff \eqref{it:consts}$ and thus the conclusion holds.
\end{proof}

To describe the above correspondence  explicitly, fix
homogeneous bases $\{e_j\,|\,1\leqslant j\leqslant l\}$  and
$\{v_i\,|\,1\leqslant i\leqslant m\}$ of   a  Lie superalgebra $\G$ and a representation $(V,\rho)$ of  $\G$  respectively. We denote by $\{e_j^*\,|\,1\leqslant j\leqslant l\}$ and $\{v_i^*\,|\,1\leqslant i\leqslant m\}$ the dual bases
  of $\G^*$ and $V^*$
respectively. Let $(sV)^*$ be the dual vector superspace of
the parity reverse vector superspace $sV$.
Then $\{sv_i\}$  forms a  homogeneous basis of $sV$ and  $\{(sv_i)^*\}$ is the dual basis of $(sV)^*$ with
$|sv_i|=|(sv_i)^*|=|v_i|+\bar 1$ for $1\leqslant i
\leqslant m$. \rxd{The description on bases has been rewritten. Please check.}

By Eq.~\eqref{eq3.2} and $T^s(sv)=T(v)$ for all $sv\in sV$, the correspondence between \eqref{it:consr} and \eqref{it:consrs} can be written explicitly as follows.

\begin{coro} \label{co:rrs}
Let $(V,\rho)$ be a representation of a Lie superalgebra $\G$ and
$T:V\longrightarrow \G$ be  a homogeneous linear map. Then \begin{equation}\label{eq:expr} r_T=\sum_{i=1}^{m} (T(v_i) \otimes
v_i^*+(-1)^{(|T|+\bar 1)(|v_i|+\bar 1)}v^*_i\otimes T(v_i))
\end{equation}
is a \sr with degree $|T|$ in the Lie superalgebra $\G \ltimes_{\rho^*}V^*$
if and only if
\begin{eqnarray}\label{eq:exprs}
r_{T^s}=\sum_{i=1}^{m} (T(v_i) \otimes
(sv_i)^*+(-1)^{|T||v_i|}(sv_i)^*\otimes T(v_i))
\end{eqnarray}
is a \sr with degree $|T|+\bar 1$ in the Lie superalgebra $\G \ltimes_{{(\rho^s)}^*}(sV)^*$.
\end{coro}
We call this correspondence between even and odd \srs in
(different) Lie superalgebras {\bf $r$-matrix duality}.

Let $T:V\longrightarrow \G$ be  a homogeneous linear map. Suppose
that $T^*$ and $(T^s)^*$ are the dual maps of the linear maps $T$
and $T^s$ respectively, that is, for all $x^*\in \G^*,v\in V$ and
$ sv \in sV,$
$$\langle T^*(x^*), v\rangle=(-1)^{|T||x^*|}\langle x^*,
T(v)\rangle,\;\; \langle (T^s)^*(x^*),
sv\rangle=(-1)^{|T^s||x^*|}\langle x^*, T^s(sv)\rangle.$$

With the notations as above, for all $1\leqslant i \leqslant m$ and $1\leqslant j \leqslant
l$, the homogeneous
linear maps $T_{r_T}: V\oplus \G^*\ra \G\oplus V^*$ and
$T_{r_{T^s}}: sV\oplus \G^*\ra \G\oplus (sV)^*$ are given
by
\begin{equation}\label{eq:28}
    T_{r_T}(v_i):=T(v_i), T_{r_T}(e_j^*):=-(-1)^{|T|} T^*(e_j^*),
\end{equation}
\begin{eqnarray}\label{eq:29}
    T_{r_{T^s}}(sv_i):=T^s(sv_i)=T(v_i), T_{r_{T^s}}(e_j^*):=
    (-1)^{|T|} (T^s)^*(e_j^*).
\end{eqnarray}
 In fact, for $r_{T^s}$ given by Eq.~(\ref{eq:exprs}), from Eqs.~ \eqref{eq:2.7} and  \eqref{eq:2.1} we obtain
\begin{equation*}
    \langle e_j^*, T_{r_{T^s}}(sv_i)\rangle=(-1)^{|T^s||sv_i|}\langle
    e_j^*\otimes sv_i, r_{T^s}\rangle
    =\sum_{k=1}^{m}(-1)^{|sv_i||sv_k|}\langle e_j^*, T(v_k)\rangle
    \langle sv_i, (sv_k)^*\rangle =\langle e_j^*, T(v_i)\rangle,
\end{equation*}
and
{\small\begin{equation*} \langle sv_i,
        T_{r_{T^s}}(e_j^*)\rangle =(-1)^{|T^s||e_j^*|}\langle sv_i\otimes
        e_j^*, r_{T^s}\rangle
        =(-1)^{|T^s||e_j^*|+|T||v_i|+|e_j^*||sv_i^*|+|sv_i|} \langle
        e_j^*, T^s(sv_i)\rangle
        =(-1)^{|T|}\langle  sv_i, (T^s)^*(e_j^*)\rangle.
\end{equation*}} We derive that $T_{r_{T^s}}(sv_i)=T(v_i), T_{r_{T^s}}(e_j^*)=
(-1)^{|T|}(T^s)^*(e_j^*)$ since the canonical pairing is
non-degenerate.
Similarly, for $r_{T}$ given by Eq.~(\ref{eq:expr}),
we have $T_{r_{T}}$ given as in Eq.~\eqref{eq:28}.

Then the correspondence between
\eqref{it:constr} and \eqref{it:constrs} is written explicitly as follows, providing another duality between  even and odd
$\OO$-operators of different Lie superalgebras, but both associated
to the coadjoint representations. This duality is called the {\bf second $\OO$-operator duality} to distinguish it from the $\OO$-operator duality in Theorem~\ref{thm:sV}.

\begin{coro}\label{coro:3.6}
Let $(V,\rho)$ be a representation of a  Lie superalgebra $\G$ and
$T:V\longrightarrow \G$ be  a homogeneous linear map.  Then the
linear map $T_{r_T}: V\oplus \G^*\ra \G\oplus V^*$ defined by
Eq.~\eqref{eq:28} is an $\mathcal O$-operator with degree $|T|$ of
the Lie superalgebra $\G \ltimes_{\rho^*}V^*$ associated to the
coadjoint representation if and only if the linear map
$T_{r_{T^s}}: sV\oplus \G^*\ra \G\oplus (sV)^*$ defined by Eq.~\eqref{eq:29} is an $\mathcal
O$-operator with degree $|T|+\bar 1$ of the Lie superalgebra $\G
\ltimes_{{(\rho^s)^*}}(sV)^*$ associated to the coadjoint
representation.
\end{coro}
Therefore, from an even or an odd $\mathcal O$-operator, one
obtains a parity pair of $\mathcal O$-operators, a parity pair of
\pansym \srs  in (different) Lie superalgebras, and a pair of
$\mathcal O$-operators  associated to the coadjoint
representations of Lie superalgebras. We use the following
diagram to summarize these relations, along with the
corresponding one for ordinary Lie algebras as comparison, extending the diagram in Eq.~\eqref{eq:dual} in the Introduction.
{\small
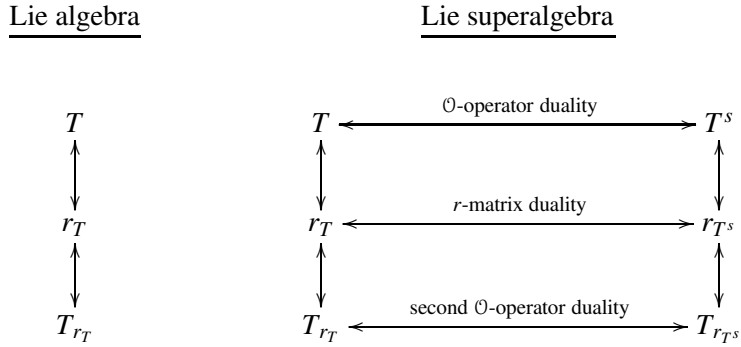
$$ \xymatrix{
    \underline{\text{Lie algebra}}      &&   &\underline{\text{Lie superalgebra}}&\\
    T\ar@{<->}[d] &&  T \ar@{<->}[d] \ar@{<->}[rr]^{\OO \text{-operator duality}} &&T^s \ar@{<->}[d]   \\
    r_T\ar@{<->}[d] &&r_T \ar@{<->}[d] \ar@{<->}[rr]^{r\text{-matrix duality}}    &&r_{T^s}\ar@{<->}[d]\\
    T_{r_T} &&  T_{r_T} \ar@{<->}[rr]^{\text{second } \OO\text{-operator duality}}   &&T_{r_{T^s}}
}
$$}
Note that in the above diagram, the vertical correspondences
keep the parities, whereas
the horizontal correspondences exchange the parities and are called dualities.

\subsection{A tree family of \srs}\label{ss:tree}
A fascinating property in the
theory of the CYBE in Lie algebras is that, starting with a
skew-symmetric $r$-matrix $r$ in a Lie algebra $\G$, one obtains an
$\OO$-operator $T_r$ of $\G$ associated to the coadjoint
representation of $\G$. Then by the construction in~\cite{Bai},
the $\OO$-operator gives rise to a skew-symmetric
$r$-matrix $r_1$ to the CYBE in a semi-direct product Lie algebra
$\G_1:=\G\ltimes_\ad \G$. Repeating this procedure to $r_1$ gives
a third skew-symmetric $r$-matrix $r_2$ of the CYBE
in an even larger semi-direct product Lie algebra and eventually a
sequence of skew-symmetric $r$-matrices $r_n, n\geq
1$. We now show that a similar procedure from a \pansym
super $r$-matrix generates a ``super"
sequence of super $r$-matrices. Here a ``super" sequence
means that the generated family is multi-parameterized,
represented by an infinitely expanding tree.

Let $\{e_i\}$ be a homogeneous  basis of a  Lie superalgebra $\G$, $\{e_i^*\}$ be the dual  basis of   $\G^*$ and  $\{(e_i^*)^*\}$ be the double dual  basis of    $(\G^*)^*$ for $1\leqslant i\leqslant l$. Then  $\{se_i\}$  and $\{se_i^*\}$ are  homogeneous  bases of $s\G$ and $s\G^*$ respectively, and  we denote by   $\{(se_i^*)^*\}$   the dual basis of    $(s\G^*)^*$  with  $|se_i|=|se_i^*|=|(se_i^*)^*|=|e_i|+\bar 1$ for $1\leqslant i\leqslant l$.  Suppose that $r=\sum_{i,j=1}^{l} a_{ij}e_i\otimes e_j\in \G\otimes \G$ is a \pansym \sr.
By  Eq. ~ (\ref{eq:2.1}), the \op  $T_r:\G^*\longrightarrow \G$ has the form $T_r(e_i^*)=\sum_{j=1}^{l}(-1)^{|e_i^*|}a_{ji}e_j$  for $1\leqslant i\leqslant l$.
Take $(V, \rho)=(\G^*,\ad^*)$ in Corollary~ \ref{co:rrs}. Then
\begin{eqnarray*}
    r_{T_r}   &=&\sum_{i=1}^{l}\big(T_r(e_i^*)\otimes (e_i^*)^*+(-1)^{(|r|+\bar 1)(|e^*_i|+\bar 1)}(e_i^*)^* \otimes T_r(e_i^*)\big)\\
    &=&\sum_{i,j=1}^{l}(-1)^{|e_i^*|}a_{ji}\big(e_j\otimes (e_i^*)^*+(-1)^{(|r|+\bar 1)(|e^*_i|+\bar 1)}(e_i^*)^* \otimes e_j\big) \in \Sol_{|r|}(\G \ltimes_{(\ad^*)^*}(\G^*)^*)
\end{eqnarray*}
and
\begin{eqnarray*}
    r_{T_r^s}    &=&\sum_{i=1}^{l}\big(T_r^s(se_i^*)\otimes (se_i^*)^*+(-1)^{|e^*_i||r|}(se_i^*)^* \otimes T_r^s(se_i^*)\big)\\
    &=&\sum_{i,j=1}^{l}(-1)^{|e_i^*|}a_{ji}\big(e_j\otimes (se_i^*)^*+(-1)^{|e^*_i||r|}(se_i^*)^* \otimes e_j\big) \in  \Sol_{|r|+\bar 1}(\G \ltimes_{((\ad^*)^s)^*}(s\G^*)^*).
\end{eqnarray*}

Applying Lemma~\ref{prop:2.4} to the module isomorphisms
$$\theta: \G\ra (\G^*)^*, \quad
\quad e_i\mapsto (-1)^{|e_i|}(e_i^*)^*, $$
and
$$\varphi: (s\G^*)^*\ra
s\G, \quad (se_i^*)^*\mapsto se_i, $$
we obtain isomorphisms between semi-direct
product Lie superalgebras
$$\G \ltimes_{(\ad^*)^*}(\G^*)^*\cong \G
\ltimes_{\ad}\G, \quad \G \ltimes_{((\ad^*)^s)^*}(s\G^*)^*\cong \G
\ltimes_{\ad^s} s\G.$$
In summary, we obtain

\begin{prop} \label{pp:rpair}
    Let $\G$ be a   Lie superalgebra. If  $r=\sum_{i,j=1}^{l} a_{ij}e_i\otimes e_j\in \G\otimes \G$ is a \pansym \sr, then there are two pan-supersymmetric \srs in the semi-direct product Lie superalgebras
    \begin{equation} \label{eq:4.9}
        r_{T_r}=\sum_{i,j=1}^{l}a_{ji}\big((e_j,0)\otimes (0,e_i)+(-1)^{(|r|+\bar 1)(|e_i|+\bar 1)}(0,e_i)\otimes (e_j,0)\big)\in \Sol_{|r|}(\G \ltimes_{\ad}\G),
    \end{equation}
    and
    \begin{equation} \label{eq:4.10}
        r_{T_r^s}=\sum_{i,j=1}^{l}(-1)^{|e_i|}a_{ji}(e_j\otimes se_i+(-1)^{|e_i||r|}se_i \otimes e_j)\in  \Sol_{|r|+\bar 1}(\G \ltimes_{\ad^s} s\G),
      \end{equation}
where  $\{e_i\}$ and $\{se_i\}$ are homogeneous  bases of    $\G$  and $s\G$ respectively for  $1\leqslant i\leqslant l$. Here we use $(e_j,0)$  and $(0, e_i)$ to denote the elements in Lie superalgebra $\G$ and in the adjoint module $\G$ respectively to avoid confusion.
\end{prop}

Thus starting with  a \pansym \sr, Proposition~\ref{pp:rpair}
gives a pair of new pan-supersymmetric super
$r$-matrices in the Lie superalgebras $\G \ltimes_{\ad}\G$ and $\G
\ltimes_{\ad^s} s\G$.
 This can be illustrated in the
following diagram.
$$ \xymatrix{
    r \ar@{->}[rr] && T_r \ar@{->}[rr]  \ar@{<->}[d]_{\OO\text{-operator}}^{\text{duality}} &&  r_{T_r}  \ar@{<->}[d]_{r\text{-matrix}}^{\text{duality}} \\
     &&T_r^s \ar@{->}[rr]
   &&r_{T_r^s} }
$$

Now iterating the same construction to this pair of new super $r$-matrices, we obtain a whole hierarchy of super $r$-matrices in a hierarchy of Lie superalgebras as follows.

To simplify the notations, for any Lie superalgebra $\G$ and any $r\in \Sol_{|r|}(\G)$, we abbreviate
$$ \G_+\coloneqq \G\ltimes_{\ad}\G, \quad \G_-\coloneqq \G\ltimes_{\ad^s} s\G$$
and \vspace{-.1cm}
$$r_+\coloneqq r_{T_r}\in \Sol_{|r|}(\G_+), \quad r_-\coloneqq r_{T_r^s}\in \Sol_{|r|+\bar{1}}(\G_-)$$
from Proposition~\ref{pp:rpair}, and utilize  similar notations
for the further derived Lie superalgebras and super $r$-matrices.   Thus for example
$$ \G_{+-}\coloneqq (\G_{+})_-\coloneqq \G_+ \ltimes_{\ad^s} s\G_+
= (\G\ltimes _{\ad} \G) \ltimes_{\ad^s} s(\G\ltimes_{\ad} \G)$$
and
$$ r_{+-}\coloneqq (r_+)_-\coloneqq r_{T_{r_+}^s}
= r_{T_{r_{T_r}}^s}$$
is in $\Sol_{|r|+\bar{1}}(\G_{+-})$.
With these abbreviations, from one super $r$-matrix $r\in \Sol_{|r|}(\G)$, we obtain the following tree hierarchy of \pansym super $r$-matrices.

\vspace{-.5cm}

\tikzstyle{level 1}=[level distance=2cm, sibling distance=3.2cm]
\tikzstyle{level 2}=[level distance=4.2cm, sibling distance=1.6cm]
\tikzstyle{level 3}=[level distance=2.7cm, sibling distance=0.8cm]
\tikzstyle{bag} = [text width=7em, text centered]
\tikzstyle{end} = [circle, minimum width=3pt,fill, inner sep=0pt]

\begin{equation}
\begin{split}
\begin{tikzpicture}[->,>=stealth',grow=right]
\node[bag] { $r  \in\Sol_{|r|}(\G)$}
    child {
        node[bag] {$r_{-}\in\Sol_{|r|+\bar 1}(\G_{-})$}
            child {
                node[bag]
                    {$r_{--} \in\Sol_{|r|}(\G_{--})$}
                child {
                node[label=right:
                    {$r_{---} \in\Sol_{|r|+\bar 1}(\G_{---})\cdots$}] {}
                edge from parent}
                child {
                node[label=right:
                    {$r_{--+} \in\Sol_{|r|}(\G_{--+})\cdots$}] {}
                edge from parent}
                 edge from parent}
            child {
                node[bag]
                    {$r_{-+} \in\Sol_{|r|+\bar 1}(\G_{-+})$}
                 child {
                node[label=right:
                    {$r_{-+-} \in\Sol_{|r|}(\G_{-+-})\cdots$}] {}
                edge from parent}
               child {
                node[label=right:
                    {$r_{-++} \in\Sol_{|r|+\bar 1}(\G_{-++})\cdots$}] {}
                edge from parent}
            edge from parent }
            edge from parent
          }
    child {
        node[bag] {$r_{+} \in\Sol_{|r|}(\G_{+})$}
        child {
                node[bag]
                    {$r_{+-} \in\Sol_{|r|+\bar 1}(\G_{+-})$}
                child {
                node[label=right:
                    {$r_{+--} \in\Sol_{|r|}(\G_{+--})\cdots$}] {}
                edge from parent}
                child {
                node[label=right:
                    {$r_{+-+} \in\Sol_{|r|+\bar 1}(\G_{+-+})\cdots$}] {}
                edge from parent}
                 edge from parent}
            child {
                node[bag]
                    {$r_{++} \in\Sol_{|r|}(\G_{++})$}
                 child {
                node[label=right:
                    {$r_{++-} \in\Sol_{|r|+\bar 1}(\G_{++-})\cdots$}] {}
                edge from parent}
               child {
                node[label=right:
                    {$r_{+++} \in\Sol_{|r|}(\G_{+++})\cdots$}] {}
                edge from parent}
            edge from parent }
        edge from parent
           };
\end{tikzpicture}
\end{split}
\label{di:tree}
\end{equation}
\vspace{-.2cm}

In fact, more generally, starting with either an even or an odd
$\OO$-operator $T$, Theorem~\ref{thm:sV} gives a pair of even and
odd $\OO$-operators $T$ and $T^s$, which then gives a pair of
pan-supersymmetric super $r$-matrices. To each of
these two super $r$-matrices, we can iterate the above
construction and again  get a similar tree hierarchy of super $r$-matrices.
\vspace{-.1cm}
\begin{exam}\label{ex:3.17}
    Consider  the Lie superalgebra $\G$ and the super $r$-matrices given in  Example \ref{exam:4.4}.
    The non-zero products in Lie superalgebra $\G\ltimes_{\ad}\G$ are given by
    \begin{eqnarray*}
        [(e,0), (f,0)]=(f,0), [(e,0), (0,f)]=(0,f), [(f,0), (0,e)]=(0,-f),
    \end{eqnarray*}
while  Lie superalgebra $\G\ltimes_{\ad^s} s\G$  is defined by non-zero products
    \begin{eqnarray*}
        [e, f]=f, [e, sf]=sf, [f, se]=sf.
    \end{eqnarray*}

    By Eqs. ~(\ref{eq:4.9}) and (\ref{eq:4.10}), starting with  the even skew-supersymmetric super $r$-matrix $r_{0}=f\otimes f$, we obtain both an even super $r$-matrix and an odd super $r$-matrix as follows:
    $$r_{T_{r_0}}=(f,0)\otimes (0,f)+(0,f)\otimes (f,0)  \in \Sol_{\bar 0}(\G \ltimes_{\ad}\G), $$
   $$ r_{T^s_{r_0}}=-f\otimes sf-sf\otimes f\in \Sol_{\bar 1}(\G \ltimes_{\ad^s} s\G).$$

    Similarly, starting with  the odd supersymmetric super $r$-matrix $r_{1}=e\otimes f+f\otimes e$, we obtain another pair of super $r$-matrices:
    $$r_{T_{r_1}}=(f,0)\otimes (0,e)+(e,0)\otimes (0,f)+(0,e)\otimes (f,0)+(0,f)\otimes (e,0)\in \Sol_{\bar 1}(\G \ltimes_{\ad}\G), $$
    $$r_{T^s_{r_1}}=-e\otimes sf+f\otimes se+sf\otimes e+se\otimes f\in \Sol_{\bar 0}(\G \ltimes_{\ad^s} s\G).$$
\end{exam}

\subsection{Self-reversing representations}\label{ss:selfrev}
Now we apply our general construction to the representations that
are isomorphic to their parity reverses.

\begin{defn}
    A representation $(V,\rho)$ of a Lie superalgebra $\G$ is called \textbf{self-reversing} if it is isomorphic to its \parityrev.
\end{defn}
A self-reversing representation can be constructed from any representation  of a Lie superalgebra $\G$ as follows.
\begin{lemma} \label{lem:selfrev}
    Let $(V,\rho)$ be a representation of a Lie superalgebra $\G$. Then the direct sum representation $(V\oplus sV, \rho+\rho^s)$ of $\G$ is self-reversing.
\end{lemma}

\begin{proof}
For all $z\in \G$ and $v\in V, su\in sV$, we have
\begin{eqnarray*}
(\rho+\rho^s)^s(z)(su, v)=(-1)^{|z|}(s(\rho(z)u), s(\rho^s(z)sv))=(\rho^s(z)su, \rho(z)v)=(\rho^s+\rho)(z)(su, v).
\end{eqnarray*}
The direct sum representation $\rho+\rho^s$  is isomorphic to $\rho^s+\rho$ and then isomorphic to $(\rho+\rho^s)^s$, that is,  the representation $(V\oplus sV, \rho+\rho^s)$ of $\G$ is self-reversing.
\end{proof}

\begin{exam}\label{exam:2.3}
    We consider  the  Lie superalgebra $\G=\frak{sl}(1|1)$ with a homogeneous basis $\{e_1, f_1, f_2\}$, where $e_1$  is even,  $f_1, f_2$ are odd and   the non-zero product is $[f_1, f_2]=e_1.$

    Let $\{v_1, v_2, w_1, w_2\}$ be a homogeneous basis of a $2|2$-dimensional vector superspace $V$ with $v_1, v_2$ even, $w_1, w_2$ odd.    Define a  representation  $(V,\rho)$  of $\G$   by taking $\rho(e_1)$ to be the identity map and
    \begin{eqnarray*}
        &&\rho(f_1) v_1=0,  \rho(f_1) v_2=w_2,  \rho(f_1) w_1=v_1,  \rho(f_1) w_2=0, \\
        &&\rho(f_2)v_1=w_1,  \rho(f_2) v_2=0,  \rho(f_2) w_1=0,  \rho(f_2) w_2=v_2.
    \end{eqnarray*}
    Then the $\G$-module $V$ is isomorphic to the direct sum of   $\G$-modules $V_1=\mathrm{span}\{v_1, w_1\}$ and $V_2=\mathrm{span}\{v_2, w_2\}$. Since $V_2$ is isomorphic to $sV_1$, Lemma~  \ref{lem:selfrev} yields that the representation $(V, \rho)$ is self-reversing.
\end{exam}

Any $\mathcal O$-operator of a Lie superalgebra $\G$ associated to
a representation $(V,\rho)$ can be extended to an $\mathcal
O$-operator of  $\G$ associated to the self-reversing
representation $(V\oplus sV, \rho+\rho^s)$. Explicitly,

\begin{coro}\label{coro:3.13}
    Let $(V,\rho)$ be a representation of a Lie superalgebra $\G$. For each
     $T\in \OO_{|T|}(\G;V,\rho)$, define a homogeneous linear map $\Ext{T}:V\oplus sV\rightarrow \G$ by
    \begin{equation}
        \Ext{T}(v, su):=T(v), \;\;\forall v\in V, su\in sV.
    \end{equation}
    Then $|\Ext{T}|=|T|$ and $\Ext{T}$ is in $\OO_{|T|}(\G;V\oplus sV,\rho+\rho^s)$.
\end{coro}

\begin{proof}
It is obvious that $|\Ext{T}|=|T|$. For all $v_1, v_2 \in V$ and $su_1, su_2\in sV$, we have
\begin{eqnarray*}
&&[\Ext{T}(v_1, su_1),\Ext{T}(v_2, su_2)]\\
&=&[T(v_1), T(v_2)] \\
&=& T((-1)^{(|T|+|v_1|)|T|}\rho(T(v_1))v_2-
(-1)^{|v_1|(|T|+|v_2|)}\rho(T(v_2))v_1)\\
&=&\Ext{T}\Big((-1)^{(|T|+|v_1|)|T|}\rho(T(v_1))v_2-
(-1)^{|v_1|(|T|+|v_2|)}\rho(T(v_2))v_1, \\
&&\quad \;\;\;\;(-1)^{(|T|+|v_1|)|T|}\rho^s(T(v_1))su_2-
(-1)^{|v_1|(|T|+|v_2|)}\rho^s(T(v_2))su_1\Big)\\
&=&\Ext{T}((-1)^{(|T|+|v_1|)|T|}(\rho+\rho^s)(T(v_1))(v_2, su_2)-
(-1)^{|v_1|(|T|+|v_2|)}(\rho+\rho^s)(T(v_2))(v_1,su_1))\\
&=&\Ext{T}((-1)^{(|T|+|v_1|)|T|}(\rho+\rho^s)(\Ext{T}(v_1, su_1))(v_2, su_2)-
(-1)^{|v_1|(|T|+|v_2|)}(\rho+\rho^s)(\Ext{T}(v_2, su_2))(v_1,su_1)).
\end{eqnarray*}
Hence $\Ext{T}\in \OO_{|T|}(\G;V\oplus sV,\rho+\rho^s)$.
\end{proof}

\begin{coro}
    Let $(V,\rho)$ be a representation of a Lie superalgebra $\G$. For each
     $T\in \OO_{|T|}(\G;V,\rho)$, we have
    $$(\Ext{T})^s=\Ext{T^s}.$$
\end{coro}

\begin{proof}
It follows from Theorem ~\ref{thm:sV} and Corollary ~\ref{coro:3.13} that $(\Ext{T})^s, \Ext{T^s}:sV\oplus V \ra \G$ are both in $\OO_{|T|+\bar 1}(\G;sV\oplus V,\rho^s+\rho)$ and
\begin{eqnarray*}
(\Ext{T})^s(su, v)=\Ext{T}(u, sv)=T(u)=T^s(su)=\Ext{T^s}(su,v), \;\;\forall v\in V, su\in sV.
\end{eqnarray*} Hence the conclusion holds.
\end{proof}

Specialized to self-reversing representations, Theorem
\ref{thm:sV} and Lemma~\ref{prop3.2} imply that odd
$\OO$-operators and even $\OO$-operators of a Lie superalgebra
associated to the same representation can be obtained
from each other by parity duality.

\begin{coro}\label{coro:3.9}
    For a self-reversing representation  $(V, \rho)$ of a Lie superalgebra $\G$, we have $\OO_{|T|+\bar 1}(\G;V,\rho)=\{\widetilde{T}|\widetilde{T}=T^s\phi, T\in\OO_{|T|}(\G;V,\rho)\}$, where $T^s$ is obtained by $T$ via Eq.~\eqref{eq:3.10} and $\phi: V\longrightarrow sV$ is an isomorphism between $(V, \rho)$ and $(sV, \rho^s)$.
\end{coro}

\begin{exam}\label{ex:3.7}
    Let $(V, \rho)$ be the self-reversing representation  of the Lie superalgebra $\G=\frak{sl}(1|1)$ in Example \ref{exam:2.3}. If a linear map $T:V\ra \G$ is  an even $\mathcal O$-operator of $\G$ associated to
    $(V,\rho)$, then a direct and lengthy computation shows that $T$ is in one of the following three forms:
    \begin{enumerate}
        \item $T_1(v_1)=k_1e_1, T_1(v_2)=k_2e_1, T_1(w_1)=0, T_1(w_2)=0, k_1\neq 0;$
        \item $T_2(v_1)=0, T_2(v_2)=k_2e_1, T_2(w_1)=0, T_2(w_2)=0, k_2\neq 0;$
        \item $T_3(v_1)=l_2e_1, T_3(v_2)=l_3e_1, T_3(w_1)=l_1f_1+l_2f_2, T_3(w_2)=l_3f_1+l_4f_2, l_1l_4=l_2l_3,$
    \end{enumerate}
    where $k_i,  l_j\in \F$ for  $i=1,2, j=1,2,3,4$.
    The invertible even linear map $\phi: V\longrightarrow sV$ defined by
    \begin{equation}\label{eq:33}
        \phi(v_1)=-sw_2, \phi(v_2)=sw_1,  \phi(w_1)=sv_2,  \phi(w_2)=-sv_1
    \end{equation}
    is an isomorphism between the representations $(V, \rho)$ and $(sV, \rho^s)$.
    It follows from Corollary \ref{coro:3.9} that  $\widetilde{T}=T^s\phi\in \OO_{\bar 1}(\G;V,\rho)$ is in one of the following three forms:
    \begin{enumerate}
        \item $\widetilde{T}_1(v_1)=0, \widetilde{T}_1(v_2)=0, \widetilde{T}_1(w_1)=k_2e_1, \widetilde{T}_1(w_2)=-k_1e_1, k_1\neq 0;$
        \item $\widetilde{T}_2(v_1)=0, \widetilde{T}_2(v_2)=0, \widetilde{T}_2(w_1)=k_2e_1, \widetilde{T}_2(w_2)=0, k_2\neq 0;$
        \item $\widetilde{T}_3(v_1)=-l_3f_1-l_4f_2, \widetilde{T}_3(v_2)=l_1f_1+l_2f_2, \widetilde{T}_3(w_1)=l_3e_1, \widetilde{T}_3(w_2)=-l_2e_1, l_1l_4=l_2l_3.$
    \end{enumerate}
\end{exam}

Set $\I_0:=\{1,\cdots,m\}$, $\I_1:=\{m+1, m+2,\cdots,2m\}$ and  $\I:=\I_0 \cup \I_1$. Let $\{v_i|i\in \I\}$ be a homogeneous basis of an $m|m$-dimensional vector superspace $V$   and $\{v_i^*|i\in \I\}$ be the dual basis of $V^*$ with  $|v_i|=|v_i^*|=\bar 0$  if $i \in \I_0$ and $|v_i|=|v_i^*|=\bar 1$ if $i\in\I_1$.
  Then $\{sv_i|i\in \I\}$   forms a  homogeneous basis of $sV$ and we denote by $\{(sv_i)^*|i\in \I\}$  the dual basis of  $(sV)^*$ with  $|sv_i|=|(sv_i)^*|=|v_i|+\bar 1$ for $i\in \I$.

\begin{theorem}\label{prop:3.10}
Let  $(V,\rho)$ be an $m|m$-dimensional self-reversing representation of a Lie superalgebra $\G$ and  $T\in \OO_{|T|}(\G;V, \rho)$.
Suppose that $\phi: V\longrightarrow sV$ is an isomorphism between the representations $(V,\rho)$ and $(sV,\rho^s)$
and $\phi^*: (sV)^*\longrightarrow V^*$  is its dual map. Then $r_{T^s}$  in Corollary~ \ref{co:rrs} is a \sr with degree $|T|+\bar 1$ in the Lie superalgebra $\G \ltimes_{\rho^*}V^*$ and has the form
\begin{eqnarray}\label{eq:r-s}
r_{T^s}&=&\sum_{i\in \I}\big(T(v_i)\otimes \phi^*((sv_i)^*)+(-1)^{|T||v_i|}\phi^*((sv_i)^*)\otimes T(v_i)\big)\nonumber\\
&=&\sum_{i\in \I}(T^s\phi(v_i)\otimes v_i^*+(-1)^{|T|+|T||v_i|}v_i^*\otimes T^s\phi(v_i)).
\end{eqnarray}
\end{theorem}

\begin{proof}
With the notations as above, for all $i \in \I_0, j \in \I_1$, we set
$$\phi(v_i)=\sum_{k\in \I_1}p_{k-m,i}sv_k, \;\;\;\;\; \phi(v_j)=\sum_{l\in \I_0}q_{l,j-m}sv_l,$$
with coefficients in the base field $\F$. Then the even linear map $\phi^*: (sV)^*\longrightarrow V^*$ defined by
$$\phi^*((sv_i)^*)=\sum_{k\in \I_1}q_{i,k-m}v^*_k, \;\;\;\;\; \phi^*((sv_j)^*)=\sum_{l\in \I_0}p_{j-m,l}v^*_l, \;\;\forall i \in \I_0, j \in \I_1,$$
 is an isomorphism between the representations $((sV)^*, (\rho^s)^*)$ and $(V^*, \rho^*)$. Hence we obtain an isomorphism between the semi-direct product Lie superalgebras  $\G \ltimes_{(\rho^s)^*}(sV)^*\cong \G \ltimes_{\rho^*}V^*$. By the isomorphisms $\phi$ and $\phi^*$, we derive that the  \sr $r_{T^s}$  in the Lie  superalgebra $\G \ltimes_{\rho^*} V^*$ has the form
 \begin{eqnarray*}
r_{T^s}&=&\sum_{i\in \I}(T(v_i)\otimes \phi^*((sv_i)^*)+(-1)^{|T||v_i|}\phi^*((sv_i)^*)\otimes T(v_i))\\
&=&\sum_{i\in \I}(T^s(sv_i)\otimes \phi^*((sv_i)^*)+(-1)^{(|T^s|+\bar 1)(|(sv_i)^*|+\bar 1)}\phi^*((sv_i)^*)\otimes T^s(sv_i))\\
&=&\sum_{i\in \I_0, k\in \I_1}(T^s(q_{i,k-m}sv_i)\otimes v_k^*+v_k^*\otimes T^s(q_{i,k-m}sv_i))\\
&&+\sum_{l\in \I_0, j\in \I_1}(T^s(p_{j-m,l}sv_j)\otimes v_l^*-(-1)^{|T^s|}v_l^*\otimes T^s(p_{j-m,l}sv_j))
\\
&=&\sum_{k\in \I_1}(T^s(\phi(v_k))\otimes v_k^*+v_k^*\otimes T^s(\phi(v_k)))+\sum_{l\in \I_0}(T^s(\phi(v_l))\otimes v_l^*-(-1)^{|T^s|}v_l^*\otimes T^s(\phi(v_l)))\\
&=&\sum_{i\in \I}(T^s\phi(v_i)\otimes v_i^*+(-1)^{|T|+|T||v_i|}v_i^*\otimes T^s\phi(v_i)).
\end{eqnarray*}
This completes the proof.
\end{proof}

Thus under the self-reversing assumption, both an even super
$r$-matrix and an odd super $r$-matrix can be obtained in the same
Lie superalgebra $\G \ltimes_{\rho^*} V^*$:

\begin{coro}\label{cor:equ}
Let  $(V,\rho)$ be a self-reversing representation of a Lie
superalgebra $\G$  and $T:V\longrightarrow \G$ be  a homogeneous
linear map. Suppose that $\phi: V\longrightarrow sV$ is an
isomorphism between the representations $(V,\rho)$ and $(sV,\rho^s)$.
 Then the following statements are equivalent.
\begin{enumerate}
\item $T$ is in $\OO_{|T|}(\G;V, \rho)$; \item
$\widetilde{T}=T^s\phi$ is in $\OO_{|T|+\bar 1}(\G;V, \rho)$;
\item $r_T$ given by $T$ via Eq.~\eqref{eq3.2} is in
$\Sol_{|T|}(\G \ltimes_{\rho^*}V^*)$; \item $r_{T^s}$
given by $\widetilde{T}=T^s\phi$ via Eq.~\eqref{eq:r-s} is in
$\Sol_{|T|+\bar 1}(\G \ltimes_{\rho^*}V^*)$; \item $T_{r_T}$ given
by $r_T$  via Eq.~ \eqref{eq:2.1} is in $\OO_{|T|}(\G
\ltimes_{\rho^*}V^*; (\G \ltimes_{\rho^*}V^*)^*, \ad^*)$; \item
$T_{r_{T^s}}$ given by  $r_{T^s}$  via Eq.~
\eqref{eq:2.1} is in $\OO_{|T|+\bar 1}(\G \ltimes_{{\rho}^*}V^*;
(\G \ltimes_{{\rho}^*}V^*)^*,\ad^*)$. \label{it:equf}
\end{enumerate}
\end{coro}

\begin{proof}
 It follows by Theorems ~\ref{cor:cons} and
\ref{prop:3.10} and Corollary ~\ref{coro:3.9}.
\end{proof}

 Moreover, \op $T_{r_{T^s}}: V\oplus \G^*\ra \G\oplus V^*$ in
\eqref{it:equf} can be given more explicitly as follows.
By a discussion similar to Eq.~(\ref{eq:29}),
the \op $T_{r_{T^s}}$ given by  $r_{T^s}$ in Eq.~\eqref{eq:r-s} has the form
\begin{eqnarray}
T_{r_{T^s}}(v_i)=T^s\phi(v_i),
T_{r_{T^s}}(e_j^*)= (-1)^{|T|}\phi^*(T^s)^*(e_j^*),
\end{eqnarray}
 where $(T^s)^*$ is the  dual map of $T^s$, $\{v_i\}$ and $\{e^*_j\}$ are  homogeneous  bases of  $V$ and $\G^*$.

Under the assumption that the representation $(V,\rho)$ is self-reversing, we revisit the tree hierarchy of \srs  in the diagram in Eq.~\eqref{di:tree}.
For a given homogeneous $\OO$-operator $T$ in
$\OO_{|T|}(\G;V, \rho)$ in Corollary~\ref{cor:equ}, we obtain a
parity pair $r_T$ and $r_{T^s}$ of super $r$-matrices in the
same Lie superalgebra $\G\ltimes_{\rho^*} V^*$. Now starting from
this parity pair of super $r$-matrices in place of the single
super $r$-matrix $r$ in the construction of the tree hierarchy in
the diagram in Eq.~\eqref{di:tree}, at
each recursive step, we can replace the super $r$-matrix by a
parity pair of super $r$-matrices, again in the same Lie
superalgebra. Continuing this way, we obtain a tree hierarchy of
parity pairs of super $r$-matrices with the same tree structure as
in the diagram in Eq.~\eqref{di:tree},
but with each super $r$-matrix there replaced by a parity pair of
super $r$-matrices in the same Lie superalgebra.

\begin{exam}
Consider the Lie superalgebra $\G=\frak{sl}(1|1)$ and  the even $\mathcal O$-operator  $T_3$  of $\G$ associated to the self-reversing representation $(V,\rho)$ in Example ~ \ref{ex:3.7}.
By the isomorphism $\phi: V\longrightarrow sV$ given in Eq. ~(\ref{eq:33}), we obtain an isomorphism $\phi^*$ between the representations $((sV)^*, (\rho^s)^*)$ and $(V^*, \rho^*)$ defined by
\begin{equation*}
\phi^*((sv_1)^*)=-w_2^*, \phi^*((sv_2)^*)=w_1^*,  \phi^*((sw_1)^*)=v_2^*,  \phi^*((sw_2)^*)=-v_1^*.
\end{equation*}

It follows from Eqs.~\eqref{eq:expr} and  ~ \eqref{eq:r-s} that
\begin{eqnarray*}
r_T&=&l_2e_1\otimes v_1^*+ l_3e_1\otimes v_2^*+(l_1f_1+l_2f_2)\otimes w_1^*+(l_3f_1+l_4f_2)\otimes w_2^* \\
&&-v_1^*\otimes l_2e_1-v_2^*\otimes l_3e_1+w_1^*\otimes (l_1f_1+l_2f_2)+ w_2^*\otimes (l_3f_1+l_4f_2), \\
r_{T^s}&=&-l_2e_1\otimes w_2^*+ l_3e_1\otimes w_1^*+(l_1f_1+l_2f_2)\otimes v_2^*-(l_3f_1+l_4f_2)\otimes v_1^* \\
&&-w_2^*\otimes l_2e_1+w_1^*\otimes l_3e_1+v_2^*\otimes (l_1f_1+l_2f_2)- v_1^*\otimes (l_3f_1+l_4f_2)
\end{eqnarray*}
with $l_1l_4=l_2l_3$, are even and odd \srs in a common  Lie superalgebra $\G \ltimes_{\rho^*} V^*$, which is defined by the following non-zero products
\begin{eqnarray*}
        &&[e_1, v_1^*]=-v_1^*,  [e_1, v_2^*]=-v_2^*,  [e_1, w_1^*]=-w_1^*,  [e_1, w_2^*]=-w_2^*, \\
        &&[f_1, v_1^*]=-w_1^*,   [f_1, w_2^*]=v_2^*,  [f_2, v_2^*]=-w_2^*,   [f_2, w_1^*]=v_1^*,  [f_1, f_2]=e_1.
    \end{eqnarray*}
\end{exam}

We end this section by giving  a general construction of a parity pair of \srs in the same Lie superalgebra.
\begin{coro}
Let  $(V,\rho)$ be a representation of a Lie superalgebra $\G$ and  $T\in \OO_{|T|}(\G;V, \rho)$.
Then $r_T:=\mT-(-1)^{|T|}\sigma(\mT)$ and   $r_{T^s}:=\mT^s-(-1)^{|T^s|}\sigma(\mT^s)$ 
are a pair of even and odd \srs in the Lie superalgebra $\G \ltimes_{\rho^*+(\rho^s)^*}(V^*\oplus (sV)^*)$.
\end{coro}

\begin{proof}
For the self-reversing representation $(V \oplus sV, \rho+\rho^s)$ and $T\in \OO_{|T|}(\G;V,\rho)$, we have $\Ext{T}\in \OO_{|T|}(\G;V \oplus sV,
\rho+\rho^s)$ by Corollary~ \ref{coro:3.13} and $\widetilde{\Ext{T}}=(\Ext{T})^s\phi\in
\OO_{|T|+1}(\G;V\oplus sV, \rho+\rho^s)$ by Corollary~ \ref{coro:3.9}, where the isomorphism $\phi:V\oplus sV\ra sV\oplus V$ is given by $\phi(v,0)=(0,v), \phi(0,su)=(su,0)$ for all $v\in V$ and $su\in sV$. It follows from Corollary~\ref{cor:equ} that $r_{\Ext{T}}=r_T=\mT-(-1)^{|T|}\sigma(\mT)$ and $r_{\Ext{T}^s}=r_{T^s}=\mT^s-(-1)^{|T^s|}\sigma(\mT^s)$ are a pair of even and odd  \srs  in the Lie superalgebra $\G \ltimes_{\rho^*+(\rho^s)^*}(V^*\oplus (sV)^*)$.
\end{proof}

This conclusion can also be obtained from Theorem~\ref{cor:cons} and the fact that both $\G\ltimes_{\rho^*}V^*$ and $\G\ltimes_{(\rho^s)^*}(sV)^*$ are Lie subsuperalgebras of $\G \ltimes_{\rho^*+(\rho^s)^*}(V^*\oplus (sV)^*)$.

\section{Pre-Lie superalgebras and $\mathcal{O}$-operators of Lie superalgebras}
\label{sec:prelie}

We show that pre-Lie superalgebras provide an easy supply of both even and odd $\mathcal O$-operators of the sub-adjacent Lie superalgebras.
Consequently, any pre-Lie superalgebra gives rise to  both even and odd \srs  in (different) semi-direct product Lie superalgebras.

Pre-Lie superalgebras (also called left-symmetric superalgebras) are the $\Z_2$-graded version of pre-Lie algebras and have appeared in many  fields of mathematics and mathematical physics; see \cite{Bu,Man,TBGS,ZB} for the background and examples.

Let $\mathcal{A}=\mathcal{A}_{\bar
0}\oplus \mathcal{A}_{\bar 1}$ be a vector superspace. Suppose
that there is a binary operation on $\mathcal A$ such that
${\mathcal A}_\alpha\, {\mathcal A}_\beta  \subseteq {\mathcal
A}_{\alpha+\beta}$ for all $\alpha,\beta\in \mathbb Z_2$, then
${\mathcal A}$ is called a {\bf superalgebra}. If
${\mathcal A}_\alpha {\mathcal A}_\beta\subseteq {\mathcal
A}_{\alpha+\beta+\bar 1}$ for all $\alpha,\beta\in \mathbb Z_2$, then
the product is called {\bf odd}.

\begin{defn}
A superalgebra $\mathcal{A}=\mathcal{A}_{\bar 0}\oplus \mathcal{A}_{\bar 1}$ is called  a \textbf{pre-Lie superalgebra} if
the associator $(x,y,z):=(xy)z-x(yz)$ is supersymmetric in $x,y$, that is,
\begin{equation}
(x,y,z)=(-1)^{|x||y|}(y,x,z),
\end{equation}
or equivalently, $(xy)z-x(yz)=(-1)^{|x||y|}\big((yx)z-y(xz)\big)$ for all $x,  y, z\in \mathcal{A}$.
\end{defn}

Let $\mathcal A$ be a pre-Lie superalgebra. Then the supercommutator
\begin{equation}
[x,y]:=x y-(-1)^{|x||y|}y x,\;\;\forall x,  y\in \mathcal{A},
\end{equation}
defines a Lie superalgebra structure on $\mathcal{A}$, which is
called the \textbf{sub-adjacent Lie superalgebra} of $\mathcal{A}$
and denoted by $\G(\mathcal{A})$. In this case, $\mathcal{A}$ is
called the \textbf{compatible pre-Lie superalgebra} on
$\G(\mathcal{A})$.
  Define the left  multiplication operator $L(x)$   by $L(x)y:=xy$
for all $x,y \in \mathcal{A}. $ Then  the even
linear map $L:\G(\mathcal{A})\longrightarrow
\mathfrak{gl}(\mathcal{A})$ with  $x\mapsto L(x)$  gives rise to  a representation  $(\mathcal A, L)$ of the sub-adjacent Lie superalgebra $\G(\mathcal A)$, which is called the \textbf{left regular representation} of $\G(\mathcal A)$.

There is a close relationship between pre-Lie superalgebras and $\mathcal {O}$-operators of the sub-adjacent Lie superalgebras.

\begin{prop}\label{prop:3.22}
Let   $(V=V_{\bar 0}\oplus V_{\bar 1}, \rho)$ be   a representation of  a  Lie superalgebra   $\G$ and  $T\in \OO_{|T|}(\G;V, \rho)$. Define a product $\cdot$ on  the vector superspace $V$ by
\begin{equation}\label{eq:3.23}
v\cdot w:=(-1)^{|T|(|v|+|T|)}\rho(T(v))w, \;\;\forall v, w \in V.
\end{equation}
This product satisfies $V_\alpha \cdot V_\beta \subseteq V_{\alpha+\beta+|T|}$ for $\alpha, \beta\in \Z_2$ and
 \begin{equation}\label{eq:3.24}
(v, w, u)=(-1)^{(|v|+|T|)(|w|+|T|)}(w, v, u), \;\;\forall u, v, w\in V.
\end{equation}
When $T$ is even, Eq.~\eqref{eq:3.23} defines a pre-Lie superalgebra  on $V$. When $T$ is odd, the product $\cdot$ on $V$ is an odd product. Further, for the product $\circ$ on $sV$ defined by
\begin{equation}\label{eq:3.20}
sv\circ sw:=s(v\cdot w),\;\;  \forall v, w\in V,
\end{equation}
the pair $(sV,\circ)$ is a pre-Lie superalgebra.
\end{prop}

\begin{proof}
For all $u, v, w\in V$,  it follows from Eq. ~(\ref{eq:3.23}) that
\begin{eqnarray*}
&&(v, w, u)=(v\cdot w)\cdot u-v\cdot(w\cdot u)\\
&=&(-1)^{|T|(|v|+|T|)}(\rho(T(v)) w)\cdot u-(-1)^{|T|(|w|+|T|)} v\cdot (\rho(T(w))u)\\
&=&(-1)^{|T|(|v|+|T|)+|T|(|v|+|w|)}\rho(T(\rho(T(v))w))u-(-1)^{|T|(|w|+|T|)+|T|(|v|+|T|)}\rho(T(v))\rho(T(w))u\\
&=&(-1)^{|T|(|w|+|T|)}\rho(T(\rho(T(v))w))u-(-1)^{|T|(|w|+|v|)}\rho(T(v))\rho(T(w))u.
\end{eqnarray*}
Similarly, we have
\begin{eqnarray*}
(w, v, u)
=(-1)^{|T|(|v|+|T|)}\rho(T(\rho(T(w))v))u -(-1)^{|T|(|w|+|v|)}\rho(T(w))\rho(T(v))u.
\end{eqnarray*}
Since $T\in \OO_{|T|}(\G;V, \rho)$,  we derive
\begin{eqnarray*}
&&(v, w, u)-(-1)^{(|v|+|T|)(|w|+|T|)}(w, v, u)\\
&=&(-1)^{|w||v|+|T|}(\rho(T(w))\rho(T(v))u-(-1)^{(|w|+|T|)(|v|+|T|)}\rho(T(v))\rho(T(w))u)\\
&&-(-1)^{(|v|+|T|)|w|}\rho(T(\rho(T(w))v))u+(-1)^{|T|(|w|+|T|)}\rho(T(\rho(T(v))w))u\\
&=&(-1)^{|w||v|+|T|}\rho\Big([T(w), T(v)]-(-1)^{|T|(|w|+|T|)}T(\rho(T(w))v)+(-1)^{|w|(|v|+|T|)}T(\rho(T(v))w)\Big)u\\
&=&0.
\end{eqnarray*}
When $T$ is even, it is obvious that Eq.~(\ref{eq:3.23}) defines a pre-Lie superalgebra  on $V$.
When $T$ is odd, we have $V_\alpha \cdot V_\beta \subseteq V_{\alpha+\beta+\bar 1}$. For all $sv\in (sV)_\alpha$ and $sw\in (sV)_\beta$, we have $v\in V_{\alpha-\bar 1}$ and $w\in V_{\beta-\bar 1}$. Thus $sv\circ sw =s(v\cdot w) \in (sV)_{\alpha+\beta}$, which implies $(sV)_{\alpha} \circ (sV)_{\beta} \subseteq (sV)_{\alpha+\beta}$. Furthermore, for all $su, sv, sw \in sV$, we have
\begin{eqnarray*}
&&(sv\circ sw)\circ su-sv\circ (sw\circ su)\\
&=&s((v\cdot w)\cdot u-v\cdot(w\cdot u))\\
&=&(-1)^{(|v|+\bar 1)(|w|+\bar 1)}s((w\cdot v)\cdot u-w\cdot(v\cdot u))\\
&=&(-1)^{|sv||sw|}((sw\circ sv)\circ su-sw\circ (sv\circ su)).
\end{eqnarray*}
Hence Eq.~\eqref{eq:3.20} defines a pre-Lie superalgebra structure on $sV$.
\end{proof}

\begin{remark}In fact, there are other odd
structures such as the Schouten bracket~\cite{JSc}, also
known as odd Poisson bracket~\cite{Ku2} and antibracket~\cite{GPS}, which is a prototypical odd product defined on
alternating multivector fields and makes them into a Gerstenhaber
algebra~\cite{Ge}.
\end{remark}

\begin{prop}\label{prop:3.23}
Let $(V, \rho)$ be a representation of a Lie superalgebra $\G$.
If $T$ is an odd (resp. even) $\mathcal {O}$-operator  of $\G$ associated to
$(V,\rho)$ and $T^s$ is  given by Eq.~\eqref{eq:3.10}, then
the pre-Lie superalgebras on $sV$ (resp. $V$)  obtained from $T$ and $T^s$ by Proposition~\ref{prop:3.22}  are the same.
\end{prop}

\begin{proof}
If $T$ is an odd $\mathcal {O}$-operator  of  $\G$ associated to
$(V,\rho)$, then $T^s$  is  an even  $\mathcal {O}$-operator  of  $\G$ associated to  $(sV,\rho^s)$ by Theorem~ \ref{cor:cons}.
It follows from Proposition \ref{prop:3.22} that  for all  $sv, sw\in sV$,
the pre-Lie superalgebras on $sV$ obtained from $T$ and $T^s$ via
$$sv\circ sw:=s(v\cdot w)=(-1)^{|v|+\bar 1}s(\rho(T(v))w)$$
and
$$sv\cdot sw:=\rho^s(T^s(sv))sw=\rho^s(T(v))sw=(-1)^{|v|+\bar 1}s(\rho(T(v))w)$$
are  the same. The even case can be similarly proved.
\end{proof}

\begin{coro}\label{coro:3.22}
Let $T$ be an  $\mathcal {O}$-operator  of  $\G$ associated to
$(V,\rho)$. Then there is an induced pre-Lie superalgebra   on the image $T(V)\subseteq \G$ of $T$ given by
\begin{equation}\label{eq:39}
T(v)\ast T(w):=T(v\cdot w), \;\;\forall v, w \in V,
\end{equation}
where the product $\cdot$ is defined in Eq.~\eqref{eq:3.23}.
\end{coro}

\begin{proof}
We first show that Eq.~\eqref{eq:39} is well-defined. We only need to verify that, if $T(v)=0$, then $T(v\cdot w)=T(w\cdot v)=0$ for all $w\in V$. In fact, if $T(v)$=0, then for any $w\in W$, we have
\begin{eqnarray*}
T(v\cdot w)&=&T\big((-1)^{|T|(|v|+|T|)}\rho(T(v))w\big)=0,\\
T(w\cdot v)&=& T\big((-1)^{|T|(|w|+|T|)}\rho(T(w))v\big)=T((-1)^{|T|(|w|+|T|)}\rho(T(w))v-(-1)^{|w|(|T|+|v|)}\rho\big(T(v))w\big)\\
&=&[T(w),T(v)]=0.
\end{eqnarray*}
Hence  Eq.~\eqref{eq:39} is well-defined.

For all $u, v, w \in V$, we have
\begin{eqnarray*}
&&(T(v)\ast T(w))\ast T(u)-T(v)\ast (T(w)\ast T(u))\\
&=&T((v\cdot w)\cdot u-v\cdot(w\cdot u))\\
&=&(-1)^{(|v|+|T|)(|w|+|T|)}T((w\cdot v)\cdot u-w\cdot(v\cdot u))\\
&=&(-1)^{(|v|+|T|)(|w|+|T|)}((T(w)\ast T(v))\ast T(u)-T(w)\ast (T(v)\ast T(u))).
\end{eqnarray*}
Hence the  conclusion holds.
\end{proof}

\begin{theorem}\label{prop:2.3}
Let  $(V, \rho)$ be a representation of
 a  Lie superalgebra $\G$ and $T\in \OO_{|T|}(\G;V, \rho)$. If $T$ is  invertible, then there exists a compatible pre-Lie
superalgebra structure on $\G$ defined by
\begin{equation}\label{eq:4.3}
x\ast y:=(-1)^{|T||x|}T\big(\rho(x)T^{-1}(y)\big),\;\;\forall x, y \in \G.
\end{equation}
Moreover, this pre-Lie
superalgebra coincides with the pre-Lie superalgebra obtained from the invertible $\OO$-operator $T^s$.

Conversely, if there exists a compatible pre-Lie superalgebra
structure $\mathcal A$ on $\G$, then the identity map ${\rm id}$
on $\G$ is an even $\mathcal O$-operator of $\G$ associated to the
left regular representation $(\mathcal A, L)$ of $\G=\G(\mathcal A)$.
Moreover, ${\rm id}^s$ is an odd
$\mathcal O$-operator of $\G$ associated to the representation
$(s{\mathcal A}, L^s)$.
\end{theorem}

\begin{proof}
Since $T$ is   invertible,  for all $x, y\in \G$, there exist $v, w\in
V$ such that $x=T(v), y=T(w)$ and $|x|=|T|+|v|, |y|=|T|+|w|$. It follows from Proposition~ \ref{prop:3.22} and Corollary ~\ref{coro:3.22} that there exists an induced pre-Lie superalgebra on $\G=T(V)$ given by
\begin{eqnarray*}
x\ast y= T(v)\ast T(w)=T(v \cdot w)=(-1)^{|T|(|v|+|T|)}T(\rho(T(v))w)=(-1)^{|T||x|}T(\rho(x)T^{-1}(y)).
\end{eqnarray*}
By the assumption that $T$ is in $\OO_{|T|}(\G;V, \rho)$, we have
\begin{equation*}
[x,y]=[T(v),T(w)]
= T((-1)^{(|T|+|v|)|T|}\rho(T(v))w- (-1)^{|v|(|T|+|w|)}\rho(T(w))v)
=x\ast y-(-1)^{|x||y|}y\ast x.
\end{equation*}
Hence Eq.~(\ref{eq:4.3}) defines a compatible pre-Lie superalgebra structure on $\G$.

By  Theorem~ \ref{cor:cons}, the linear map $T^s\in \OO_{|T|+\bar 1}(\G; sV, \rho^s)$ is also invertible.
Then there exists another compatible pre-Lie superalgebra structure $\ast_1$ on $\G$ obtained from $T^s$.  For $x=T(v)=T^s(sv)$ and $y=T(w)=T^s(sw)$,  we have
\begin{eqnarray*}
x\ast_1 y&=&T^s(sv)\ast_1 T^s(sw)=(-1)^{|T^s|(|sv|+|T^s|)}T^s(\rho^s(T^s(sv))sw)\\
&=&(-1)^{|T^s|(|sv|+|T^s|)+|T^s|+|sv|}T^s(s(\rho(T^s(sv))w))
=(-1)^{|T|(|v|+|T|)}T(\rho(T(v))w)\\
&=& (-1)^{|T||x|}T(\rho(x)T^{-1}(y))=x\ast y.
\end{eqnarray*}

Conversely, it is obvious that the identity map ${\rm id}$ is an
even $\mathcal O$-operator of $\G$ associated to $(\mathcal A, L)$
as in the Lie algebra case. Then the last statement follows
from Theorem~\ref{cor:cons}.
\end{proof}

We illustrate the above general results by the following example.
\begin{exam}\label{ex:3.20}
Let $\G=\G_{\bar 0}\oplus \G_{\bar 1}$ be the $1|1$-dimensional Lie superalgebra with
a homogeneous basis $\{e, f\}$, $e \in \G_{\bar 0}$ and $f\in \G_{\bar 1}$, whose non-zero product is given by $[f, f]=-2e.$
Let $V=V_{\bar 0}\oplus V_{\bar 1}$ be a $1|1$-dimensional vector superspace with a homogeneous basis $\{v, w\}$,  $v \in V_{\bar 0}$ and $w\in V_{\bar 1}$.  Consider the representation $\rho: \G \longrightarrow \Gl(V)$ defined by
 \begin{eqnarray*}
\rho(e) v=v,  \rho(e) w=w,  \rho(f) v=w,  \rho(f) w=-v.
\end{eqnarray*}
The linear map $T: V \longrightarrow \G$ given by
$T(v)=f$ and $T(w)=e$
 is an invertible odd   $\mathcal {O}$-operator  of  $\G$ associated to the representation $(V,\rho)$.
It follows from Proposition ~\ref{prop:3.22} that there is an odd product $\cdot$ on $V$ obtained from $T$ defined by
$$v\cdot v=-w, v\cdot w=v, w\cdot v=v, w\cdot w=w.$$
Then it gives a pre-Lie superalgebra $\A_V$ on $sV$:
$$sv\circ sv=-sw, sv\circ sw=sv, sw\circ sv=sv, sw\circ sw=sw.$$
Furthermore, by
Theorem~\ref{cor:cons}, the linear map
$T^s:sV\longrightarrow \G$ given by $T^s(sv)=f, T^s(sw)=e$ is an
invertible even $\mathcal O$-operator of $\G$ associated to the
\parityrev representation $(sV,\rho^s)$. It gives the same pre-Lie
superalgebra $\A_V$ on $sV$ by Eq. ~\eqref{eq:3.23}. Moreover, by
Theorem ~\ref{prop:2.3}, we obtain a compatible pre-Lie
superalgebra $\A$ on $\G$ from the invertible   $\mathcal
{O}$-operator $T$ given by
 $$e\ast e=e,  e\ast f=f, f\ast e=f, f\ast f=-e.$$
\end{exam}

Now let $\A$ be a $k|l$-dimensional pre-Lie superalgebra. Suppose that $\{e_1,\cdots, e_k, f_1, \cdots, f_l\}$ is a homogeneous basis of $\A$ and $\{e_1^*,\cdots, e_k^*, f^*_1,\cdots, f^*_l\}$ is  the  dual basis of $\A^*$  with $|e_i|=|e^*_i|=\bar  0$ and  $|f_j|=|f^*_j|=\bar 1$.
We use $\{sf_1, \cdots, sf_l, se_1, \cdots, se_k\}$ and  $\{(sf_1)^*, \cdots, (sf_l)^*, (se_1)^*, \cdots, (se_k)^*\}$ to denote a basis of $s\A$ and the dual basis of $(s\A)^*$ respectively with $|sf_j|=|(sf_j)^*|=\bar  0$ and  $|se_i|=|(se_i)^*|=\bar 1$ for $1\leqslant i\leqslant k$ and $1\leqslant j\leqslant l$.

By Theorem~\ref{prop:2.3} and Corollary~ \ref{co:rrs}, we have the following result.

\begin{prop}\label{coro:4.5}
With the above notations, a pre-Lie superalgebra $\A$ induces a
pair of even and  odd \srs  in a pair of semi-direct product Lie
superalgebras respectively as follows.
\begin{enumerate}
\item {\rm (\cite{WHB})}
$r_{\id}=\sum_{i=1}^k (e_i \otimes e_i^*-e^*_i\otimes e_i) +\sum_{j=1}^l(f_j\otimes f^*_j+ f_j^*\otimes f_j)\in \Sol_{\bar 0}(\G(\A) \ltimes_{L^*} \A^*);$
\item
$r_{\id^s}=\sum_{i=1}^k (e_i \otimes (se_i)^*+(se_i)^*\otimes e_i) +\sum_{j=1}^l (f_j\otimes (sf_j)^*+ (sf_j)^*\otimes f_j)\in \Sol_{\bar 1}(\G(\A) \ltimes_{(L^s)^*} (s\A)^*).$
\end{enumerate}
\end{prop}

Therefore from any pre-Lie superalgebra $\A$,
one obtains a pair of even and odd \srs in the semi-direct product Lie superalgebras
 $\G(\A) \ltimes_{L^*} \A^*$ and $\G(\A) \ltimes_{(L^s)^*} (s\A)^*$ respectively.
In general,  these Lie superalgebras are not isomorphic.

However, if the left regular representation $(\mathcal A, L)$ of
the sub-adjacent Lie superalgebra $\G(\A)$ of a pre-Lie
superalgebra $\A$ is self-reversing, then the Lie superalgebras
$\G(\A) \ltimes_{L^*} \A^*$ and $\G(\A) \ltimes_{(L^s)^*} (s\A)^*$
will be isomorphic. In this case, from a pre-Lie superalgebra
$\A$, a parity pair of \pansym super $r$-matrices is
obtained in the same Lie superalgebra $\G(\A) \ltimes_{L^*} \A^*$.
We supply such an example to finish the paper.
\vspace{-.1cm}
\begin{exam}
Let $\A$ be the compatible  pre-Lie superalgebra  on the Lie superalgebra $\G$ given in Example~\ref{ex:3.20}.
It is clear that the invertible even linear map $\phi: \A\ra s\A$ defined by $\phi(e)=sf, \phi(f)=se$ gives
an isomorphism between the representations $(\mathcal A, L)$ and
$(s\mathcal A, L^s)$ of the sub-adjacent Lie superalgebra $\G=\G(\A)$
of the pre-Lie superalgebra $\A$, that is, the left regular representation
$(\mathcal A, L)$ of  $\G=\G(\A)$ is self-reversing. Hence
the Lie superalgebra $\G \ltimes_{(L^s)^*} (s\A)^*$  is isomorphic to $\G \ltimes_{L^*} \A^*$,
whose  non-zero products are
$$[f, f]=-2e,  [e,e^*]=-e^*,  [e,f^*]=-f^*,  [f,e^*]=f^*,  [f,f^*]=e^*,$$
where $\{e^*, f^*\}$ is the dual basis of $\A^*$.
By Proposition~\ref{coro:4.5}, the element
\begin{equation*}
r_{\id}=e \otimes e^*-e^*\otimes e +f\otimes f^*+ f^*\otimes f
\end{equation*} is an even skew-supersymmetric \sr in the Lie superalgebra  $\G
\ltimes_{L^*} \A^*$. Moreover, applying Theorem
~\ref{prop:3.10} to this example, we find that
$ r_{\id^s}= e \otimes f^* +f^*\otimes e+f\otimes
e^* + e^*\otimes f
$
is an odd supersymmetric \sr in the same Lie superalgebra  $\G
\ltimes_{L^*} \A^*$.
\end{exam}

{\bf Acknowledgments.} This work is supported by
 National Natural Science Foundation of China (11931009, 12271265, 12261131498), the Fundamental Research Funds for the Central Universities and Nankai Zhide Foundation. The authors thank the referee for helpful suggestions. 

\smallskip

\noindent
{\bf Declaration of interests. } The authors have no conflicts of interest to disclose.

\smallskip

\noindent
{\bf Data availability. } Data sharing is not applicable to this article as no new data were created or analyzed in this study.


\begin{thebibliography}{999}

\bibitem{ABB} H. Albuquerque, E. Barreiro and S. Benayadi,  Odd-quadratic Lie superalgebras.  {\em J. Geom. Phys.} \textbf{60} (2010)  230-250.

\bibitem{An} N. Andruskiewitsch,  Lie superbialgebras and Poisson-Lie supergroups. {\em Abh. Math. Sem. Univ. Hamburg} \textbf{63} (1993) 147-163.

\bibitem{Bai}  C. Bai, A unified algebraic approach to the clasical Yang-Baxter equation. {\em J. Phys. A: Math. Theor.} \textbf{40} (2007) 11073-11082.

\bibitem{BGN} C. Bai, L. Guo and X. Ni, Nonabelian generalized Lax pairs, the classical Yang-Baxter equation and
PostLie algebras. \emph{Comm. Math. Phys.} {\bf 297} (2010) 553-596.

\bibitem{BB}  E. Barreiro and S. Benayadi,  Quadratic symmplectic Lie superalgebras and Lie bi-superalgebras. {\em J. Algebra}  \textbf{321} (2009) 582-608.

\bibitem{BS}  N. Beisert and E. Spill, The classical $r$-matrices of AdS/CFT and its Lie bialgebras structure. {\em Commun. Math. Phys.} \textbf{285} (2009) 537-565.

\bibitem{BD} A.~A. Belavin and V. G. Drinfeld, Solutions of the classical Yang-Baxter equation
for simple Lie algebras. {\em Funct. Anal. Appl.} {\bf 16} (1982) 159-180.


\bibitem{Bu}  D. Burde, Left-symmetric algebras, or pre-Lie
algebras in geometry and physics. {\em Cent. Eur. J. Math.} \textbf{4}
(2006) 323-357.

\bibitem{CP}  V. Chari and A. Pressley, A guide to quantum groups.  Cambridge University Press, Cambridge, 1994.

\bibitem{Ch}  S. -J. Cheng,  Representations of central extensions
of differentially simple Lie superalgebras. {\em Comm. Math. Phys.}  \textbf{154} (1993) 555-568.

\bibitem{CW} S. -J. Cheng and W. Wang, Dualities and Representations of Lie Superalgebras. {\em Graduate Studies in Math.} \textbf{144}, American Mathematical Society, Providence, RI, 2012.

\bibitem{CK}  A. Connes, D. Kreimer, Renormalization in quantum field theory and the Riemann-Hilbert
problem. I.
{\em Comm. Math. Phys.} {\bf 210} (2000) 249–273.

\bibitem{Dr}  V. Drinfeld, Hamiltonian structure on the Lie groups, Lie bialgebras and the geometric sense of the classical Yang-Baxter equations. {\em Soviet Math. Dokl.}  \textbf{27} (1983)  68-71.

\bibitem{ER} A. Eghbali, A. Rezaei-Aghdam,  The $\Gl(1|1)$ Lie superbialgebras.  {\em J. Geom. Phys.} \textbf{65} (2013)  7-25.


\bibitem{Geer} N. Geer, Etingof-Kazhdan quantization of Lie superbialgebras. {\em Adv. Math.} \textbf{207} (2006) 1-38.

\bibitem{Ge}  M. Gerstenhaber, The cohomology structure of an associative ring. {\em Ann. Math.}
\textbf{78} (1963)  267-288.

\bibitem{GPS}  J. Gomis, J. Par\'{\i}s and S. Samuel,  Antibracket, antifields and gauge-theory quantization.  {\em Phys. Rep.} \textbf{259} (1995)  1-145.

\bibitem{GZB}  M. D. Gould, R. B. Zhang and A. J. Bracken,  Lie bi-superalgebras and
the graded classical Yang-Baxter equation. {\em Rev. Math. Phys.}
\textbf{3} (1991)  223-240.

\bibitem{Ji} M. Jimbo (ed.), Yang-Baxter equation in integrable systems. {\em Adv. Ser. Math. Phys.} \textbf{10}, World Sci. Publ., Teaneck, NJ, 1990.

\bibitem{Ka1} G. Karaali, Constructing $r$-matrices on simple Lie superalgebras. {\em J. Algebra} \textbf{282} (2004) 83-102.


\bibitem{KT} S.~M.~Khoroshkin and V.~N.~Tolstoy, Universal $R$-matrix for quantized (super)algebra. \emph{Comm. Math. Phys.} {\bf 141} (1991) 599-617.

\bibitem{KS}Y. Kosmann-Schwarzbach,
Lie bialgebras, Poisson Lie groups and dressing transformations. Integrability of nonlinear systems, 104-170,
{\em Lecture Notes in Phys.} \textbf{495}, Springer, Berlin, 1997.

\bibitem{Ku2} B. A. Kupershmidt,  Odd and even Poisson brackets in dynamical systems.  \emph{Lett. Math. Phys.}   \textbf{9} (1985)  323-330.

\bibitem{Ku3}  B. A. Kupershmidt, What a classical $r$-matrix really is. {\em J. Nonlinear Math. Phys.} \textbf{6} (1999) 448-488.

\bibitem{LOZ} A. Levin, M. Olshanetsky and A. Zotov,
 Odd supersymmetric Kronecker elliptic function and Yang–Baxter equations. {\em J. Math. Phys.} \textbf{61}     (2020) 103504.

\bibitem{Man} D. Manchon, A short survey on pre-Lie algebras. In: Noncommutative Geometry and Physics: Renormalisation, Motives, Index Theory, 89-102, {\em ESI Lect. Math. Phys.}, Eur. Math. Soc., Z\"{u}rich, 2011.


\bibitem{STS} M. A. Semenov-Tian-Shansky, What is a classical
$r$-matrix? {\em Funct. Anal. Appl.} {\bf 17} (1983) 259-272.

\bibitem{Sc} M. Scheunert,  The theory of Lie superalgebras: an introduction. {\em Lecture Notes in Math.} \textbf{716}, Springer, Berlin, 1979.

\bibitem{JSc} J. A. Schouten,  \"Uber differentialkomitanten zweier kontravarianter Gr\"ossen. {\em Nederl. Akad. Wetensch. Proc.} \textbf{43} (1940) 449-452.

\bibitem{SZ}  Y. Su and R. B. Zhang, Mixed cohomology of Lie superalgebras. {\em J. Algebra}    \textbf{549} (2020)  1-29.

\bibitem{TBGS} R. Tang, C. Bai, L. Guo and Y. Sheng, Deformations and their controlling cohomologies of $\OO$-operators. {\em Comm. Math. Phys.} \textbf{368} (2019) 665-700.

\bibitem{V} T. Voronov,  Graded manifolds and Drinfeld doubles for Lie bialgebroids. {\em Contemp. Math.} \textbf{315} (2002) 131-168.

\bibitem{WHB}  Y. Wang, D. Hou and C. Bai, Operator forms of the classical Yang-Baxter equation in Lie superalgebras.
 {\em Int. J. Geom. Methods Mod. Phys.}  \textbf{7} (2010)  583-597.

\bibitem{ZB}  R. Zhang and C. Bai, On some left-symmetric superalgebras. {\em J. Algebra  Appl.}   \textbf{11} (2012)  1250097.

\bibitem{ZGB}  R. B. Zhang, M. D. Gould and A. J. Bracken, Solutions of the graded classical Yang-Baxter equation and integrable models. {\em J. Phys. A: Math. Gen.}  \textbf{24} (1991) 1185-1197.

\end{thebibliography}
\end{document}